\documentclass{imanum}
\usepackage[latin9]{inputenc}
\usepackage{amsmath}
\usepackage{amssymb}
\usepackage{amsfonts}
\usepackage{graphicx}
\usepackage{nicefrac}

\newcommand{\Th}{\mathcal{T}_h}
\usepackage{amsopn}

\begin{document}

\title{Time and space adaptivity of the wave equation discretized in time by a second order scheme} 
\shorttitle{Time and space adaptivity of the wave equation}

\author{
  Olga Gorynina,\thanks{Laboratoire de Math\'{e}matiques de Besan\c{c}on, Univ. Bourgogne Franche-Comt\'{e}, 16 route de Gray, 25030 Besan\c{c}on Cedex, France. Email: olga.gorynina@univ-fcomte.fr.}
  Alexei Lozinski,\thanks{Laboratoire de Math\'{e}matiques de Besan\c{c}on, Univ. Bourgogne Franche-Comt\'{e}, 16 route de Gray, 25030 Besan\c{c}on Cedex, France. Email: alexei.lozinski@univ-fcomte.fr.}
  and
  Marco Picasso\thanks{Institute of Mathematics, Ecole Polytechnique F\'{e}d\'{e}rale de Lausanne, Station 8, CH 1015, Lausanne, Switzerland. Email: marco.picasso@epfl.ch.}
}
\shortauthorlist{O. Gorynina, A. Lozinski, and M. Picasso}

\maketitle

\begin{abstract}
{
The aim of this paper is to obtain a posteriori error bounds of optimal order in time and space for the linear second-order wave equation discretized by the Newmark scheme in time and the  finite element method in space. Error estimate is derived in the $L^{\infty}$-in-time/energy-in-space norm. Numerical experiments are reported for several test cases and confirm equivalence of the proposed estimator and the true error.
}
{  
a posteriori error bounds in time and space, wave equation, Newmark scheme
}
\end{abstract}

\section{Introduction}

A posteriori error analysis of finite element approximations for partial differential equations plays an important role in mesh adaptivity techniques. The main aim of a posteriori error analysis is to obtain suitable error estimates computable using only the approximate solution given by the numerical method. The cases of elliptic and parabolic problems are well studied in the literature (for the parabolic case, we can cite, among many others \cite{ ErikssonJohnson, AMN, LPP, LakkisMakridakisPryer}). On the contrary, the a posteriori error analysis for hyperbolic equations  of second order in time is much less developed. Some a posteriori bounds are proposed in \cite{BS, Georgoulis13} for the wave equation using the Euler discretization in time, which is however known to be too diffusive and thus rarely used for the wave equation. More popular schemes, i.e. the leap-frog and cosine methods, are studied in \cite{Georgoulis16} but only the error caused by discretization in time is considered. On the other hand, error estimators for the space discretization only are proposed in \cite{Picasso10, adjerid2002posteriori}. Goal-oriented error estimation and adaptivity for the wave equation were developed in \cite{bangerth2010adaptive, bangerth2001adaptive, bangerth1999finite}.

The motivation of this work is to obtain a posteriori error estimates of optimal order in time and space for the fully discrete wave equation in energy norm discretized with the Newmark scheme in time (equivalent to a cosine method as presented in \cite{Georgoulis16}) and  with finite elements in space. We adopt the particular choice for the parameters in the Newmark scheme, namely $\beta=1/4 $, $\gamma=1/2$. This choice of parameters is popular since it provides a conservative method with respect to the energy norm, cf.  \cite{bathe1976numerical}. Another interesting feature of this variant of the method, which is in fact essential for our analysis, is the fact that the method can be reinterpreted as the Crank-Nicolson discretization of the reformulation of the governing equation in the first-order system, as in \cite{Baker76}. We are thus able to use the techniques stemming  from a posteriori error analysis for the Crank-Nicolson discretization of the heat equation in \cite{LPP}, based  on a piecewise quadratic polynomial in time reconstruction of the numerical solution.   This leads to
optimal a posteriori error estimate in time and also allows us to easily  recover the estimates in space. The resulting estimates are referred to as the $3$-point estimator since our quadratic reconstruction is drawn through the values of the discrete solution at 3 points in time. The reliability of 3-point estimator is proved theoretically for general regular meshes in space and non-uniform meshes in time. It is also illustrated by numerical experiments. 

We do not provide a proof of the optimality (efficiency) of our error estimators in space ans time. However, we are able to prove that the time estimator is of optimal order  at least on sufficiently smooth solutions, quasi-uniform meshes in space and uniform meshes in time. The most interesting finding of this analysis is the crucial importance of the way in which the initial conditions are discretized (elliptic projections): a straightforward discretization, such as the nodal interpolation, may ruin the error estimators while providing quite acceptable numerical solution.
Numerical experiments confirm these theoretical findings and demonstrate that our error estimators are of optimal order in space and time, even in situation not accessible to the current theory (non quasi-uniform meshes, not constant time steps). This gives us the hope that our estimators can be used to construct an adaptive algorithm in both time and space. 

The outline of the paper is as follows.  We present the governing equations, the discretization  and a priori error estimates in Section \ref{section2}. In Section \ref{section3}, an a posteriori error estimate is derived and some considerations concerning the optimality of time estimators are given. Numerical results are analysed in Section \ref{section4}.


\section{The Newmark scheme for the wave equation and a priori error analysis}\label{section2}

We consider initial boundary-value problem for the wave equation. Let $\Omega$ be a bounded domain in $\mathbb{R}^2$ with boundary $\partial \Omega$ and $T>0$ be a given final time. Let $u=u(x,t) : \Omega\times\left[0,T \right]\to\mathbb{R}$ be the solution to
\begin{equation}
   \begin{cases}
      \cfrac{\partial^2 u}{\partial t^2}-\Delta u=f,&\mbox{in}~ \Omega\times\left]0,T     \right],\\
      u=0,&\mbox{on}~ \partial\Omega\times\left]0,T \right],\\
      u(\cdot,0)=u_0,&\mbox{in}~\Omega,\\
      \cfrac{\partial u}{\partial t}(\cdot,0)=v_0,&\mbox{in}~\Omega,
   \end{cases}
   \label{wave}
\end{equation}
where $f,u_0,v_0$ are given functions. Note that if we introduce the auxiliary unknown $v=\frac{\partial u}{\partial t}$ then model (\ref{wave}) can be rewritten as the following first-order in time system 
\begin{equation}
   \begin{cases}
      \cfrac{\partial u}{\partial t}-v=0,  &\mbox{in}~ \Omega\times\left]0,T \right], \\
      \cfrac{\partial v}{\partial t}-\Delta u=f,  &\mbox{in}~ \Omega\times\left]0,T \right], \\
      u=v=0,&\mbox{on}~ \partial\Omega\times\left]0,T \right]  ,\\
      u(\cdot,0)=u_0,~v(\cdot,0)=v_0,~&\mbox{in}~\Omega.
      \label{syst}
   \end{cases}
\end{equation}
The above problem (\ref{wave}) has the following weak formulation, cf. \cite{evans2010partial}: for given  \newline 
$f\in L^{2}(0,T;L^2(\Omega))$, $u_0\in H^1_0(\Omega)$ and $v_0\in L^2(\Omega)$ find a function 
\begin{equation}
   u\in L^{2}\left(0,T;H^1_0(\Omega)\right),~\cfrac{\partial u}{\partial t}\in L^{2}\left(0,T;L^2(\Omega)\right),~\cfrac{\partial^2 u}{\partial t^2}\in L^{2}\left(0,T;H^{-1}(\Omega)\right)
\end{equation}
such that $u(x,0)=u_0$ in $H^1_0(\Omega)$, $\cfrac{\partial u}{\partial t}(x,0)=v_0$ in $L^2(\Omega)$ and
\begin{equation}
   \left\langle\cfrac{\partial^2 u}{\partial t^2},\varphi\right\rangle+\left(\nabla u,\nabla \varphi\right)=\left(f,\varphi\right),~\forall\varphi  \in H^1_0(\Omega),
   \label{weakwave}
\end{equation}
where $\left\langle \cdot, \cdot \right\rangle$ denotes the duality pairing between $ H^{-1}(\Omega)$ and $ H^1_0(\Omega)$ and the parentheses $( \cdot, \cdot)$ stand for the inner product in
$L^2 ( \Omega)$. 
Following Chap. 7, Sect. 2, Theorem 5 from \cite{evans2010partial}, we observe that in fact 
\begin{equation*}
   u\in C^{0}\left(0,T;H^1_0(\Omega)\right),~\cfrac{\partial u}{\partial t}\in C^{0}\left(0,T;L^2(\Omega)\right),~\cfrac{\partial^2 u}{\partial t^2}\in C^{0}\left(0,T;H^{-1}(\Omega)\right).
\end{equation*}
Higher regularity results with more regular data are also available in \cite{evans2010partial}.

Let us now discretize (\ref{wave}) or, equivalently, (\ref{syst}) in space using the finite element method and in time using an appropriate marching scheme. We thus introduce a regular mesh $\mathcal{T}_h$ on $\Omega$ with triangles $K$, $\mathrm{diam}~K=h_{K}$, $h=\max_{K\in\Th}h_K$, internal edges $E\in\mathcal{E}_h$, where $\mathcal{E}_h$ represents the internal edges of the mesh $\mathcal{T}_h$  and the standard finite element space ${V}_h\subset H^1_0(\Omega ) $:
$$
V_h=\left\{v_h\in C(\bar\Omega):v_h|_K\in \mathbb{P}_1~\forall K\in\mathcal{T}_h \text{ and }v_h|_{\partial\Omega}=0\right\}.
$$
Let us also introduce a subdivision of the time interval $[0,T]$
$$0=t_0<t_1<\dots<t_N=T$$
with time steps $\tau_n=t_{n+1}-t_n$ for $n=0,\ldots,N-1$ and $\tau=\displaystyle\max_{0 \leq n \leq N-1}\tau_n$ . Following \cite{Baker76}, by applying Crank-Nicolson discretization to both equations in (\ref{syst}) we get a second order in time scheme. The fully discretized method is as follows: taking $u^0_h,v^0_h\in V_h$ as some approximations to $u_0,v_0$ compute $u^n_h,v^n_h\in V_h$ for $n=0,\ldots,N-1$ from the system
\begin{align}
\label{CNh1}
\frac{u_h^{n + 1} - u_h^n}{\tau_n} - \frac{v_h^n + v_h^{n+1}}{2} &= 0,
\\
\label{CNh2}
\left( \frac{v^{n + 1}_h - v^n_h}{\tau_n}, \varphi_h \right) + \left( \nabla \frac{u^{n + 1}_h + u^n_h}{2}, \nabla \varphi_h \right) &= \left(  \frac{f^{n+1} + f^n}{2}, \varphi_h \right), \hspace{1em} \forall \varphi_h \in V_h.
\end{align}
From here on, $f^n$ is an abbreviation for $f ( \cdot,t_n)$.

Note that we can eliminate $v_h^n$ from (\ref{CNh1})-(\ref{CNh2}) and rewrite the scheme (\ref{CNh1})-(\ref{CNh2}) in terms of $u_h^n$ only. This results in the following method:
given approximations $u^0_h, v^0_h \in V_h$ of $u_0, v_0$ compute $u^1_h \in V_h$ from
\begin{equation}
    \label{Newm1}\left( \frac{u^1_h - u^0_h}{\tau_0}, \varphi_h \right) + \left( \nabla\frac{\tau_0 (u^1_h + u^0_h)}{4}, \nabla \varphi_h \right) 
    = \left( v_h^0 +\frac{\tau_0}{4} (f^1 + f^0), \varphi_h \right), \quad \forall\varphi_h \in V_h
\end{equation}
and then compute $u^{n+1}_h \in V_h$ for $n = 1, \ldots, N-1$ from equation
\begin{align}
   \left( \frac{u_h^{n + 1} - u_h^n}{\tau_n} - \frac{u_h^n - u_h^{n -1}}{\tau_{n - 1}}, \varphi_h \right) + \left( \nabla \frac{\tau_n (u_h^{n + 1}+ u_h^n) + \tau_{n - 1} (u_h^n + u_h^{n - 1})}{4}, \nabla \varphi_h \right)&\notag\\
   =\left( \frac{\tau_n (f^{n + 1} + f^n) + \tau_{n - 1} (f^n + f^{n - 1})}{4},\varphi_h \right), \hspace{1em} \forall \varphi_h \in V_h.&
   \label{Newm2}
\end{align}
This equation is derived by multiplying (\ref{CNh2}) by $\tau_n/2$, doing the same at the previous time step, taking the sum of the two results and observing 
$$
\frac{v_h^{n+1}-v_h^{n-1}}{2}
=\frac{v_h^{n+1}-v_h^{n}}{2}+\frac{v_h^{n}-v_h^{n-1}}{2}
=\frac{u_h^{n + 1} - u_h^n}{\tau_n} - \frac{u_h^n - u_h^{n -1}}{\tau_{n - 1}}
$$
by (\ref{CNh1}).

We have thus recovered the Newmark scheme (\cite{newmark1959, RT}) with coefficients $\beta =
1/4, \gamma = 1/2$ as applied to the wave equation
(\ref{wave}). Note that the presentation of this scheme in {\cite{newmark1959}} and in
the subsequent literature on applications in structural mechanics is a little
bit different, but the present form (\ref{Newm1})-(\ref{Newm2}) can be found, for example, in {\cite{RT}}. It is easy to see that for any $u^0_h, v^0_h \in V_h  $, both schemes (\ref{CNh1})-(\ref{CNh2}) and (\ref{Newm1})-(\ref{Newm2})
provide the same unique solution $u^n_h, v^n_h \in V_h$ for $n = 1, \ldots, N$.
In the case of scheme (\ref{Newm1})-(\ref{Newm2}), $v^n_h$ can be reconstructed from $u^n_h$ recursively with the formula
\begin{equation}\label{vhform}
     v_h^{n + 1} = 2 \frac{u_h^{n + 1} - u_h^n}{\tau_n} - v_h^n.
\end{equation}

From now on, we shall use the following notations
\begin{align}\label{notation}
u_h^{n+1/2}& :=\frac{u_h^{n+1}+u_h^{n}}{2},
\quad
\partial _{n+1/2}u_h:=\frac{u_h^{n+1}-u_h^{n}}{\tau_n},
\quad
\partial _{n}u_h:=\frac{u_h^{n+1}-u_h^{n-1}}{\tau_n+\tau_{n-1}} \\
\notag \partial _{n}^{2}{u_h}&:=\frac{1}{\tau_{n-1/2}}\left(\frac{u_h^{n+1}-u_h^{n}}{\tau_n}-\frac{u_h^{n}-u_h^{n-1}}{\tau_{n-1}}\right)
\text{ with } \tau_{n-1/2}:=\frac{\tau_n+\tau_{n-1}}{2}
\end{align}
We apply this notations to all quantities indexed by a superscript, so that, for example, $f^{n+1/2}=({f^{n+1}+f^n})/{2}$.
We also denote $u (x,t_n)$, $v(x,t_n)$ by $u^n$, $v^n$ so that, for example, $u^{n+1/2}=\left({u^{n+1}+u^n}\right)/{2}=\left(u(x,t_{n+1})+u(x,t_n)\right)/{2}$. 

We turn now to a priori error analysis for the scheme (\ref{CNh1})-(\ref{CNh2}). We shall measure the error in the following norm
\begin{equation}\label{EnNorm}
u \mapsto \max_{t\in[0,T]}\left(\left\|\cfrac{\partial u}{\partial t} (t)\right\|_{L^2(\Omega)}^{2}+\left\vert u(t)\right\vert^{2}_{H^1(\Omega)}\right) ^{1/2}.
\end{equation}
Here and in what follows, we use  the notations $u(t)$ and $\cfrac{\partial u}{\partial t} (t)$ as a shorthand for, respectively, $u(\cdot,t)$ and $\cfrac{\partial u}{\partial t} (\cdot,t)$. The norms and semi-norms in Sobolev spaces $H^k(\Omega)$ are denoted, respectively, by $\|\cdot\|_{H^k(\Omega)}$ and $|\cdot|_{H^k(\Omega)}$.
We call (\ref{EnNorm}) the energy norm referring to the underlying physics of the studied phenomenon. Indeed, the first term in (\ref{EnNorm}) may be assimilated to the kinetic energy and the second one to the potential energy.

Note that a priori error estimates for scheme (\ref{CNh1})-(\ref{CNh2})  can be found in \cite{Baker76, dupont19732, RT}. We are going to construct a priori error estimates following the ideas of \cite{Baker76} but we measure the error in a different norm, namely the energy norm (\ref{EnNorm}), and present the estimate in a slightly different manner, foreshadowing the upcoming a posteriori estimates.
\begin{theorem}\label{lemma}
  Let $u$ be a smooth solution of the wave equation (\ref{wave}) and $u_{h}^{n}$, $v_{h}^{n}$ be the discrete solution of the scheme (\ref{CNh1})-(\ref{CNh2}). If $u_0\in H^2(\Omega)$, $v_0\in H^1(\Omega)$ and the approximations to the initial conditions are chosen such that
  $\| v^0_h - v_0 \|_{L^2(\Omega)}  \le Ch| v_0 |_{H^1(\Omega)} $ and
  $| u^0_h - u_0 |_{H^1(\Omega)}  \le Ch | u_0 |_{H^2(\Omega)}$, then the following a priori error estimate holds
  \begin{multline}
\label{apriori}    
\max_{0\le n \le N}\left( \left\Vert v^n_h - \cfrac{\partial u}{\partial t} (t_n)\right\Vert_{L^2(\Omega)}^2 + | u^n_h - u (t_n) |^2_{H^1(\Omega)}\right)^{1/2} \\
\leq 
    Ch\left(| v_0 |_{H^1(\Omega)} + | u_0 |_{H^2(\Omega)}\right)
    \\
    + C\sum_{n = 0}^{N - 1} \tau^2_n \left(
    \int_{t_n}^{t_{n + 1}} \left\vert \cfrac{\partial^3 u }{\partial t^3} \right\vert_{H^1(\Omega)} {dt} +
    \int_{t_n}^{t_{n + 1}} \left\Vert \cfrac{\partial^4 u }{\partial t^4} \right\Vert_{L^2(\Omega)} {dt} \right)
    \\
    + C h  \left( \int_{t_0}^{t_N} \left\vert \cfrac{\partial^2 u }{\partial t^2} \right\vert_{H^1(\Omega)} {dt}
        + \sum_{n=0}^{N}  \tau_n'\left\vert \cfrac{\partial u}{\partial t} (t_n)\right\vert_{H^2(\Omega)} 
    + \left\vert \cfrac{\partial u }{\partial t}(t_N) \right\vert_{H^1(\Omega)} + | u ( t_N) |_{H^2(\Omega)}
    \right) 
  \end{multline}
 with a constant $C>0$ depending only on the regularity of the mesh $\mathcal{T}_h$. We have set here $\tau_n'=\tau_{n-1/2}$ for $1<n<N-1$ and $\tau_0'=\tau_0$, $\tau_N'=\tau_N$.
\end{theorem}

\begin{proof} Let us introduce $e^n_u = u^n_h - \Pi_h u^n$ and $e^n_v = v^n_h - I_h v^n$ where
$\Pi_h : H^1_0 (\Omega) \to V_h$ is the $H^1_0$-orthogonal projection
operator, i.e.
\begin{equation}\label{Pih}
 \left(\nabla \Pi_h v, \nabla\varphi_h\right) = \left(\nabla v,
   \nabla\varphi_h\right), \hspace{1em} \forall v \in H^1_0 (\Omega),\hspace{1em}\forall \varphi_h \in V_h 
\end{equation}
and $\tilde I_h : H^1_0 (\Omega) \to V_h$ is a Cl{\'e}ment-type interpolation operator which is also a projection, i.e. $\tilde I_h=Id$ on $V_h$, cf. \cite{ErnGue, ScoZh}. 

Let us recall, for future reference, the well known properties of these operators (see  \cite{ErnGue}): for every sufficiently smooth function $v$ the following inequalities hold
\begin{equation}\label{Pinterp}
  | \Pi_h v |_{H^1(\Omega)}\leq | v |_{H^1(\Omega)}, 
  \hspace{1em} | v - \Pi_h v |_{H^1(\Omega)} \leq Ch | v |_{H^2(\Omega)}
\end{equation}
with a constant $C > 0$ which depends only on the regularity of the mesh.
Moreover, for all $K\in\Th$ and  $E \in \mathcal{E}_h$ we have 
\begin{equation}\label{Clement}
 \| v - \tilde I_h v \|_{L^2(K)} \leq Ch_K |v |_{H^1(\omega_K)}
\text{ and }
\| v - \tilde I_h v \|_{L^2(E)} \leq Ch_{E}^{1 / 2} | v |_{H^1(\omega_{E})}
\end{equation}
Here $\omega_K$  (resp. $\omega_E$) represents the set of triangles of $\mathcal{T}_h$ having a common vertex with triangle $K$ (resp. edge $E$) and the constant $C > 0$ depends only on the regularity of the mesh.

Observe that for $\varphi_h,\psi_h\in V_h$ the following equations hold 
  \begin{align}
   \notag\left(\nabla\partial_{n + 1 / 2} e_u,\nabla\varphi_h\right)  - \left(\nabla e_v^{n + 1 / 2},\nabla\varphi_h\right) &\\
   = -&\left(\nabla  \left( \partial_{n + 1 / 2}
     u - \tilde I_h v^{n + 1 / 2}\right),\nabla\varphi_h\right), 
     \label{ErrEqApr1}\\
    \left( \partial_{n + 1 / 2} e_v, \psi_h\right)  + \left( \nabla e_u^{n + 1 / 2}, \nabla
     \psi_h\right)  = & \left( \left( \cfrac{\partial^2 u}{\partial t^2}\right)^{n + 1 / 2} - \tilde I_h \left( \partial_{n + 1 / 2}
     v\right), \psi_h\right).
     \label{ErrEqApr2}
  \end{align}
The last equation is a direct consequence of  (\ref{CNh2}) together with the governing equation (\ref{wave}) evaluated at times $t_n$ and $t_{n+1}$. In accordance with the conventions above, we have denoted here
$$
\left(\cfrac{\partial^2 u}{\partial t^2}\right)^{n + 1 / 2}:=\frac{1}{2}\left(\cfrac{\partial^2 u}{\partial t^2}(t_n)+\cfrac{\partial^2 u}{\partial t^2}(t_{n+1})\right)
$$
Equation (\ref{ErrEqApr1}) is obtained from (\ref{CNh1}) taking the gradient of both sides, multiplying by $\nabla\varphi_h$ and integrating over $\Omega$. 

Putting $\varphi_h = e_u^{n + 1 / 2}$ and $\psi_h = e_v^{n + 1 / 2}$ and taking the sum of (\ref{ErrEqApr1})--(\ref{ErrEqApr2}) yields
   \begin{align}
   \label{ErrAprInterMed}
   \frac{| e^{n + 1}_u |_{H^1(\Omega)}^2 - | e^n_u |_{H^1(\Omega)}^2 + \| e^{n + 1}_v \|_{L^2(\Omega)}^2 - \|
     e^n_v \|_{L^2(\Omega)}^2}{2 \tau_n} = &- \left( \nabla R^n_1, \nabla e_u^{n + 1 / 2}\right) \\
     \notag &+ \left(R^n_2, e_v^{n + 1 / 2}\right)  
   \end{align}
  with
  \begin{align*}
  R^n_1 = \partial_{n + 1 / 2} u - \tilde I_hv^{n + 1 / 2} \text{ and }
  R^n_2 = \left( \cfrac{\partial^2 u}{\partial t^2}\right)^{n + 1 / 2} - \tilde I_h \left( \partial_{n + 1 / 2}v\right).
  \end{align*}
Set
\[ E^n = \left(\left|e^n_u \right|_{H^1 (\Omega)}^2 +\left\|e^n_v \right\|_{L^2
   (\Omega)}^2\right)^{1/2} \]
so that equality (\ref{ErrAprInterMed}) with Cauchy-Schwarz inequality entails
\[ \frac{(E^{n + 1})^2 - (E^n)^2}{2 \tau_n} \leq \left(|R^n_1 |_{H^1 (\Omega)}^2
   +\|R^n_2 \|_{L^2 (\Omega)}^2\right)^{1/2} \frac{E^{n + 1} + E^n}{2} 
\]  
which implies
\[ E^{n + 1} - E^n \leq
   \tau_n\left(|R^n_1 |_{H^1 (\Omega)} +\|R^n_2 \|_{L^2 (\Omega)}\right). 
\]  
Summing this over $n$ from 0 to $N - 1$ gives
\begin{align}
\label{sumeN}  (|e^N_u |_{H^1 (\Omega)}^2 +\|e^N_v \|_{L^2
  (\Omega)}^2)^{{1}/{2}} &\leq (|e^0_u |_{H^1 (\Omega)}^2 + \|e^0_v
  \|_{L^2 (\Omega)}^2)^{{1}/{2}}\\
\notag   &+ \sum_{n = 0}^{N - 1} \tau_n (|R^n_1 |_{H^1 (\Omega)} +\|R^n_2 \|_{L^2
  (\Omega)}) 
\end{align} 
We have the following estimates for $R^n_1$ and $R^n_2$
\begin{align}
\label{Rn1} |R^n_1 |_{H^1(\Omega)} &\leq C \tau_n  \int_{t_n}^{t_{n + 1}} \left\vert \cfrac{\partial^3
   u}{\partial t^3} \right\vert_{H^1(\Omega)} dt\\
\notag   &+ Ch\left(
   \left\vert \cfrac{\partial u}{\partial t} \left(t^n\right) \right\vert_{H^2(\Omega)}
   + \left\vert \cfrac{\partial u}{\partial t} \left(t^{n+1}\right) \right\vert_{H^2(\Omega)}
   \right)\\
\label{Rn2}\|R^n_2 \|_{L^2(\Omega)} &\leq C \tau_n  \int_{t_n}^{t_{n + 1}}\left\Vert \cfrac{\partial^4
   u}{\partial t^4} \right\Vert_{L^2(\Omega)} dt + C \frac{h}{\tau_n}  \int_{t_n}^{t_{n + 1}} \left\vert \cfrac{\partial^2
   u}{\partial t^2} \right\vert_{H^1(\Omega)} dt
\end{align}
The proof of (\ref{Rn1})--(\ref{Rn2}) is quite standard, but tedious. For brevity, we provide here only the proof of estimate (\ref{Rn2}): we rewrite the definition of $R^n_2$ recalling that $v = \partial u/\partial t$ and using the Taylor expansion around $t=t_{n+1/2 }$ as follows
\begin{align*}
      R^n_2 &= \frac 12\left(\cfrac{\partial^2 u}{\partial t^2} ( t_{n+1}) + \cfrac{\partial^2 u}{\partial t^2} ( t_{n})\right) - \frac {1}{\tau_n}\left(\cfrac{\partial u}{\partial t} ( t_{n+1}) - \cfrac{\partial u}{\partial t} ( t_{n})\right)\\
     & + \frac {1}{\tau_n}\left( I - \tilde I_h\right)  \left(\cfrac{\partial u}{\partial t} ( t_{n+1}) - \cfrac{\partial u}{\partial t} ( t_{n})\right)
      =  \int_{t_{n + 1 / 2}}^{t_{n + 1}} \left(
  \frac{t_{n + 1} - t}{2} - \frac{( t_{n + 1} - t)^2}{2 \tau_n} \right)\cfrac{\partial^4
   u}{\partial t^4}
  dt \\
  & - \int_{t_n}^{t_{n + 1 / 2}} \left( \frac{t_n - t}{2}
  + \frac{( t_n - t)^2}{2 \tau_n} \right) \cfrac{\partial^4
   u}{\partial t^4}dt  + \frac{1}{\tau_n}
  ( I - \tilde I_h) \int_{t_n}^{t_{n + 1}} \cfrac{\partial^2
   u}{\partial t^2} dt.
\end{align*}
Taking the $L^2(\Omega)$ norm on both sides and applying the projection error estimate (\ref{Pinterp}) in $L^2(\Omega)$ we obtain (\ref{Rn2}).

Substituting (\ref{Rn1})--(\ref{Rn2}) into (\ref{sumeN}) yields
\begin{align*}
\left(\left|e^N_u \right|_{H^1(\Omega)}^2  +\left\|e^N_v\right\|_{L^2(\Omega)}^2\right)&^{1/2} \leq
 \left(\left|e^0_u\right|_{H^1(\Omega)}^2 +\left\|e^0_v \right\|_{L^2(\Omega)}^2\right)^{1/2} \\
 &+ C \sum_{n = 0}^{N - 1} \tau^2_n  \left(
   \int_{t_n}^{t_{n + 1}} \left\vert \cfrac{\partial^3
   u}{\partial t^3} \right\vert_{H^1(\Omega)} dt  + \int_{t_n}^{t_{n + 1}} \left\Vert \cfrac{\partial^4
   u}{\partial t^4} \right\Vert_{L^2(\Omega)} dt \right) \\
   & + Ch \int_0^{t_N}  \left\vert \cfrac{\partial^2 u}{\partial t^2} \right\vert_{H^1(\Omega)} dt
    + Ch \sum_{n=0}^{N}  \tau_n'\left\vert \cfrac{\partial u}{\partial t} (t_n)\right\vert_{H^2(\Omega)}. 
\end{align*}
Applying the triangle inequality and estimate (\ref{Pinterp}) in the above inequality we get
\begin{align}
\left(\left\Vert v^N_h - \cfrac{\partial u}{\partial t} (t_N) \right\Vert_{L^2(\Omega)}^2 + \left|u^N_h - u (t_N) \right|^2_{H^1(\Omega)}\right)^{1/2} \leq \left(\left|e^N_u\right|_{H^1(\Omega)}^2 +\left\|e^N_v \right\|_{L^2(\Omega)}^2\right)^{1/2}& \notag\\+
   \left(\left\Vert \left(I-\tilde I_h\right) \frac{\partial u}{\partial t} (t_N)\right\Vert_{L^2(\Omega)}^2 + \left\vert \left(I-\Pi_h\right) u (t_N) \right\vert^2_{H^1(\Omega)}\right)^{1/2}&
\end{align}
which implies (\ref{apriori}) since we can safely assume that the maximum of the error in (\ref{apriori}) is attained at the final time $t_N$ (if not, it suffices to redeclare the time where the maximum is attained as $t_N$).
\end{proof}

\begin{remark} Estimate (\ref{apriori}) is of order $h$ in space which is due to the the presence of $H^1$ term in the norm in which we
measure the error. One sees easily that essentially the proof above gives the estimate of order $h^2$, multiplied by the norms of the exact solution in more regular spaces, if the target norm is changed to \newline $\displaystyle \max_{0 \leq n \leq N}\left\Vert v^n_h - \cfrac{\partial u}{\partial t} (t_n) \right\Vert_{L^2(\Omega)}$. One would rely then on the estimate
$$
 \left\| v - \Pi_h v \right\|_{L^2(\Omega)} \leq Ch^2 | v |_{H^2(\Omega)}
$$
for the orthogonal projection error and one would obtain
  \begin{align}\label{apriori2}
     \left\Vert v^N_h - \cfrac{\partial u}{\partial t} (t_N) \right\Vert_{L^2(\Omega)} & \leq \left\| v^0_h - v_0 \right\|_{L^2(\Omega)}^2
    +Ch^2\left| v_0 \right|_{H^2(\Omega)}\\
    \notag & + \sum_{n = 0}^{N - 1} \tau^2_n \left(
    \int_{t_n}^{t_{n + 1}} \left\vert \frac{\partial^3 u}{\partial t^3} \right\vert_{H^1(\Omega)} {dt} +
    \int_{t_n}^{t_{n + 1}} \left\Vert  \frac{\partial^4 u}{\partial t^4} \right\Vert_{L^2(\Omega)} {dt} \right)
    \\
    \notag & + C h^2  \left( \int_{t_0}^{t_N} \left\vert  \frac{\partial^2 u}{\partial t^2} \right\vert_{H^2(\Omega)} {dt}
    + \left\vert  \frac{\partial u}{\partial t}(t_N) \right\vert_{H^2(\Omega)}
    \right)
  \end{align}

\end{remark}

\section{A posteriori error estimates for the wave equation in the ``energy'' norm}\label{section3}
Our aim here is to derive a posteriori bounds in time and space for the error measured in the norm (\ref{EnNorm}). We discuss some considerations about upper bound for $3$-point time estimator.

\subsection{A 3-point estimator: an upper bound for the error}

The basic technical tool in deriving time error estimator is the piecewise quadratic (in time) reconstruction of the discrete solution, already used in \cite{LPP} in a similar context.  

\begin{definition}\label{QuadRec}
Let $u^n_h$ be the discrete solution given by the scheme (\ref{Newm2}).  Then, the piecewise quadratic reconstruction $\tilde{u}_{h\tau} (t) : [0, T] \rightarrow V_h$ is constructed as the continuous in time function that is equal on $[t_n, t_{n + 1}]$, $n\ge 1$, to the quadratic polynomial in $t$ that coincides with $u^{n + 1}_h$ (respectively $u^n_h$, $u^{n - 1}_h$) at time $t_{n+1}$ (respectively $t_n$, $t_{n - 1}$).  Moreover, $\tilde{u}_{h\tau} (t)$ is defined on $[t_0, t_{1}]$ as the quadratic polynomial in $t$ that coincides with $u^{2}_h$ (respectively $u^1_h$, $u^{0}_h$) at time $t_{2}$ (respectively $t_1$, $t_{0}$). Similarly, we introduce  piecewise quadratic reconstruction $\tilde{v}_{h\tau} (t) : [0, T] \rightarrow V_h$ based on $v^n_h$ defined by (\ref{vhform}) and $\tilde{f}_{\tau} (t) : [0, T] \rightarrow L^2(\Omega)$ based on $f(t_n,\cdot)$.  
\end{definition}

Our quadratic reconstructions $\tilde{u}_{h\tau}$, $\tilde{v}_{h\tau}$ are thus based on three points in time (normally looking backwards in time, with the exemption of the initial time slab $[t_0,t_1]$). This is why the error estimator derived in the following theorem using Definition \ref{QuadRec} will be referred to as the $3$-point estimator.

\begin{theorem}\label{lemest3}
The following a posteriori error estimate holds between the solution $u$ of the wave equation (\ref{wave}) and the discrete solution $u_h^n$ given by (\ref{Newm1})--(\ref{Newm2}) for all $ t_n,~0\leq n\leq N$ with $v_h^n$ given by (\ref{vhform}):
\begin{multline}
\left(\left\Vert v^{n}_h- \cfrac{\partial u}{\partial t} (t_n)\right\Vert_{L^2(\Omega)} ^{2}+\left\vert u^{n}_h-u(t_{n})\right\vert ^{2}_{H^1(\Omega)}\right) ^{1/2}\\
\leq\left(\left\Vert v^{0}_h-v_0\right\Vert_{L^2(\Omega)} ^{2}+\left\vert u^{0}_h-u_0\right\vert ^{2}_{H^1(\Omega)}\right) ^{1/2} \\
+\eta _{S}(t_{N})+\sum_{k=0}^{N-1}\tau_k\eta _{T}(t_{k})
+\int_0^{t_n} \|f-\tilde{f}_\tau\|_{L^2(\Omega)}dt
\label{estf}
\end{multline}
where the space indicator is defined by
\begin{align}
  \eta_S (t_k)
  &= C_1 \max_{0 \leqslant t \leqslant t_k} \Biggl[ \sum_{K \in \mathcal{T}_h}
   h_K^2  \left\Vert \frac{\partial \tilde{v}_{h \tau}}{\partial t} - \Delta \tilde{u}_{h \tau}-f
   \right\Vert_{L^2(K)}^2 + \sum_{
E \in \mathcal{E}_h}h_{E} \left|\left[n \cdot \nabla \tilde{u}_{h
   \tau}\right]\right|_{L^2(E)}^2 \Biggl]^{1/2}
\notag 
\\
\notag &+ C_2\sum_{m = 0}^{k-1} \int_{t_m}^{t_{m + 1}} \Biggl[ \sum_{K \in \mathcal{T}_h} h_K^2
  \left\Vert \frac{\partial^2 \tilde{v}_{h \tau}}{\partial t^2} - \Delta \frac{\partial
  \tilde{u}_{h\tau}}{\partial t} -\frac{\partial{f}}{\partial{t}}\right\Vert_{L^2(K)}^2
  + \sum_{
E \in \mathcal{E}_h}h_{E} \left\Vert\left[n \cdot
  \nabla \frac{\partial \tilde{u}_{h\tau}}{\partial t}\right]\right\Vert_{L^2(E)}^2
  \Biggr]^{1/2}dt
\notag 
\\
&  + C_3\sum_{m = 1}^{k-1} {\tau_{m - 1}} \left[ \sum_{K \in \mathcal{T}_h} h_K^2 \left\Vert \partial_m^2 v_h -
  \partial_{m - 1}^2 v_h \right\Vert_{L^2(K)}^2 \right]^{1/2}\label{space}
\end{align}
here  $C_1,~C_2,~C_3$ are constants depending only on the mesh regularity,
$[\cdot]$ stands for a jump on an edge $E\in\mathcal{E}_h$, and 
$\tilde{u}_{h\tau}$, $\tilde{v}_{h\tau}$ are given by Definition \ref{QuadRec}. 

The error indicator in time for $k=1,\dots,N-1$ is
\begin{equation}\label{in}
\eta_T (t_k)=\left(\frac{1}{12}\tau_{k}^2+\frac{1}{8}\tau_{k-1}\tau_{k}\right)\left(\left\vert\partial _{k}^{2}{v_h}\right\vert_{H^1(\Omega)}
+ \left\Vert \partial _{k}^{2}{f_h} - z_h^k\right\Vert_{L^2(\Omega)}^2\right)^{1/2}
\end{equation}
where $z^k_h$ is such that
\begin{equation}\label{zh}
\left(z_h^k, \varphi_h\right) =(\nabla \partial _{k}^{2}{u_h}, \nabla\varphi_h),
\quad\forall\varphi_h\in V_h
\end{equation}
and
\begin{equation}\label{in0step}
\eta_T (t_0)=\left(\frac{5}{12}\tau_{0}^2+\frac{1}{2}\tau_{1}\tau_{0}\right)\left(\left\vert\partial _{1}^{2}{v_h}\right\vert_{H^1(\Omega)}
+ \left\Vert \partial _{1}^{2}{f_h} - z^1_h\right\Vert_{L^2(\Omega)}^2\right)^{1/2}
\end{equation}
\end{theorem}
\begin{proof}
In the following, we adopt the vector notation $U (t, x)
=\begin{pmatrix}
  u (t, x)\\
  v (t, x)
\end{pmatrix}$ where $v = {\partial u}/{\partial t}$. Note that the first equation in
(\ref{syst}) implies that
\[ \left(\nabla \cfrac{\partial u}{\partial t}, \nabla \varphi\right) - (\nabla v, \nabla \varphi) = 0,
   \quad \forall \varphi \in H^1_0 (\Omega) \]
by taking its gradient, multiplying it by $\nabla \varphi$ and integrating
over $\Omega$. Thus, system (\ref{syst}) can be rewritten in the vector
notations as
\begin{equation}
  {b} \left(\cfrac{\partial U}{\partial t}, \Phi\right) + \left(\mathcal{A} \nabla U, \nabla \Phi\right) = 
  {b} (F,\Phi), \quad  \forall \Phi \in (H^1_0 (\Omega))^2 \label{ODEf}
\end{equation}
where $\mathcal{A} = \begin{pmatrix}
  0 &- 1\\
  1&~0
\end{pmatrix}$, $F =\begin{pmatrix}
  0\\
  f
\end{pmatrix}$ and
$$ {b} ( U, \Phi)
 = {b}\left(\begin{pmatrix}u\\ v\end{pmatrix}, \begin{pmatrix} \varphi\\ \psi\end{pmatrix} \right)
:= (\nabla u, \nabla \varphi) + (v, \psi)
$$
Similarly, Newmark scheme (\ref{CNh1})--(\ref{CNh2}) can be rewritten as
\begin{equation}
  {b} \left( \frac{U_h^{n + 1} - U_h^n}{\tau_n}, \Phi_h \right) +
  \left( \mathcal{A} \nabla \frac{U_h^{n + 1} + U_h^n}{2}, \nabla \Phi_h \right) = 
  {b} \left(F^{n+1/2}, \Phi_h\right), \hspace{1em} \forall \Phi_h \in V_h^2
  \label{vectorScheme}
\end{equation}
where $U_h^n = \begin{pmatrix}
  u_h^n\\
  v_h^n
\end{pmatrix}$  and $F^{n+1/2} = \begin{pmatrix}
  0\\
  f^{n + 1/2}
\end{pmatrix}$.

The a posteriori analysis relies on an appropriate residual equation for the quad\-ra\-tic reconstruction $\tilde{U}_{h\tau}=\begin{pmatrix} \tilde{u}_{h\tau}\\ \tilde{v}_{h\tau}\end{pmatrix}$. We have thus for $t \in [t_n,t_{n + 1}]$, $n = 1, \ldots, N-1$
\begin{equation}
  \tilde{U}_{h\tau} (t) = U^{n + 1}_h + (t - t_{n + 1}) \partial_{n +
  1/2} U_h + \frac{1}{2}  (t - t_{n + 1})  (t - t_n) \partial_n^2 U_h
\end{equation}
so that, after some simplifications,
\begin{multline}\label{Uode1}
{b} \left( \frac{\partial \tilde{U}_{h\tau}}{\partial t}, \Phi_h
   \right) + (\mathcal{A} \nabla \tilde{U}_{h\tau}, \nabla \Phi_h)
   ={b} \left( (t - t_{n + 1/2}) \partial_n^2 U_h + F^{n+1/2},
   \Phi_h \right) \\
   + \left( (t - t_{n + 1/2}) \mathcal{A} \nabla
   \partial_{n + 1/2} U_h + \frac{1}{2} (t - t_{n + 1}) (t - t_n)
   \mathcal{A} \nabla \partial_n^2 U_h, \nabla \Phi_h \right)
\end{multline}
Consider now (\ref{vectorScheme}) at time steps $n$ and $n - 1$. Subtracting one from another and dividing by $\tau_{n - 1/2}$ yields
$$
{b} \left(\partial_n^2 U_h, \Phi_h\right) + \left(\mathcal{A} \nabla \partial_n
   U_h, \nabla \Phi_h\right) = {b}\left( \partial_n F, \Phi_h \right) 
$$
or
$$
{b} \left(\partial_n^2 U_h, \Phi_h\right) + \left(\mathcal{A} \nabla \left(
   \partial_{n + 1/2} U_h - \frac{\tau_{n - 1}}{2} \partial_n^2 U_h
   \right), \nabla \Phi_h \right) = 
   {b}\left(\partial_n F, \Phi_h \right) 
$$
so that (\ref{Uode1}) simplifies to
\begin{multline}\label{Uode2}
 {b} \left(\frac{\partial \tilde{U}_{h\tau}}{\partial t}, \Phi_h
   \right) + \left(\mathcal{A} \nabla \tilde{U}_{h\tau}, \Phi_h\right) \\
   = \left(p_n \mathcal{A} \nabla \partial_n^2 U_h, \nabla \Phi_h\right) + 
   {b}\left( \left(t - t_{n + 1/2}\right) \partial_n F + F^{n+1/2}, \Phi_h
   \right) \\
   = \left(p_n \mathcal{A} \nabla \partial_n^2 U_h, \nabla \Phi_h\right) + 
   {b}\left( \tilde{F}_\tau - p_n \partial^2_n F, \Phi_h\right) 
\end{multline}
where
\begin{align*} p_n &= \frac{\tau_{n - 1}}{2}  (t - t_{n + 1/2}) + \frac{1}{2}
   (t - t_{n + 1})  (t - t_n), \\
  \tilde{F}_{\tau} (t) &= F^{n + 1}_h + (t - t_{n + 1}) \partial_{n +
  1/2} F + \frac{1}{2}  (t - t_{n + 1})  (t - t_n) \partial_n^2 F.
\end{align*}

Introduce the error between reconstruction $\tilde{U}_{h\tau}$ and solution
$U$ to problem (\ref{ODEf}) :
\begin{equation}
  E = \tilde{U}_{h\tau} - U
\end{equation}
or, component-wise
$$
E = \begin{pmatrix}
     E_u\\
     E_v
   \end{pmatrix} = \begin{pmatrix}
     \tilde{u}_{h\tau} - u\\
     \tilde{v}_{h\tau} - v
   \end{pmatrix}
$$
Taking the difference between (\ref{Uode2}) and (\ref{ODEf}) we obtain the residual differential equation for the error valid for $t \in [t_n, t_{n
+ 1}]$, $n = 1, \ldots, N-1$
\begin{align}  
\label{errode}{b}({\partial}_{t}E,{\Phi})+(\mathcal{A}{\nabla} E,{\nabla} {\Phi})
  &={b} \left(\cfrac{{\partial} \tilde{U}_{{\tau} h}}{\partial t}-F,{\Phi}-{\Phi}_{h}\right)+\left(\mathcal{A}{\nabla} \tilde{U}_{{\tau} h},{\nabla} ({\Phi}-{\Phi}_{h})\right)\\
\notag  +\left(p_n \mathcal{A}{\nabla} {\partial}_{n}^{2} U_{h},{\nabla} {\Phi}_{h}\right)
  &+{b}\left( \tilde{F}_\tau -F - p_n \partial^2_n F, \Phi_h\right), \hspace{1em} \forall \Phi_h \in V_h^2 
\end{align}

Now we take $\Phi = E$, $\Phi_h = \begin{pmatrix}
  \Pi_h E_u\\
  \tilde I_h E_v
\end{pmatrix}$ where $\Pi_h : H^1_0 (\Omega) \to V_h$ is the
$H^1_0$-orthogonal projection operator (\ref{Pih}) and $\tilde I_h : H^1_0 (\Omega) \to V_h$ is a
Cl{\'e}ment-type interpolation operator satisfying $\tilde I_h=Id$ on $V_h$ and (\ref{Clement}). Noting that $( \mathcal{A} \nabla E,
\nabla E) = 0$ and
$$
\left( \nabla \cfrac{\partial \tilde{u}_{h\tau}}{\partial t}, \nabla ( E_u -
\Pi_h E_u)\right) = \left( \nabla \tilde{v}_{h\tau}, \nabla \left( E_u - \Pi_h E_u\right)\right) = 0
$$
Introducing operator $A_h:V_h\to V_h$ such that
\begin{equation}\label{Ah}
\left(A_h w_h, \varphi_h\right) =(\nabla w_h, \nabla\varphi_h),
\quad\forall\varphi_h\in V_h
\end{equation} 
we get
\begin{align*}
  \left(\cfrac{{\partial} E_{v}}{\partial t},E_{v}\right)+\left({\nabla} E_{u},{\nabla} \cfrac{{\partial} E_{u}}{\partial t}\right)=\left(\cfrac{{\partial}\tilde{v}_{{\tau} h}}{\partial t}-f,E_{v}-{\Pi}_{h} E_{v}\right)+\left({\nabla} \tilde{u}_{{\tau}h},{\nabla}\left(E_{v}-\tilde I_{h} E_{v}\right)\right)\\
  + \left(p_n \left(A_h {\partial}_{n}^{2} u_{h}-{\partial}_{n}^{2}f_h\right),\tilde I_{h} E_{v}\right)-\left(p_n {\nabla}{\partial}_{n}^{2} v_{h},{\nabla}E_{u}\right)
  +\left(\tilde{f}_\tau-f,\tilde I_{h} E_{v}\right).
\end{align*}
Note that equation similar to (\ref{errode}) also holds for $t \in [t_0, t_{1}]$
\begin{align} 
\label{errorode1st}{b}({\partial}_{t}E,{\Phi})+(\mathcal{A}{\nabla} E,{\nabla} {\Phi})
  &={b} \left(\cfrac{{\partial} \tilde{U}_{{\tau} h}}{\partial t}-F,{\Phi}-{\Phi}_{h}\right)+\left(\mathcal{A}{\nabla} \tilde{U}_{{\tau} h},{\nabla} ({\Phi}-{\Phi}_{h})\right)\\
\notag  &+\left(p_1 \mathcal{A}{\nabla} {\partial}_{1}^{2} U_{h},{\nabla} {\Phi}_{h}\right)
  + {b}\left( \tilde{F}_\tau -F - p_1 \partial^2_1 F, \Phi_h\right). 
\end{align}
That follows from the definition of the piecewise quadratic reconstruction $\tilde{u}_{h\tau} (t)$ for $t \in [t_0, t_{1}]$. Integrating (\ref{errode}) and (\ref{errorode1st}) in time from 0 to some $t^{\ast}\geq t_1$ yields
\begin{align}
  \notag  & {\frac{1}{2}} \left(|E_{u}|^{2}_{H^1(\Omega)}+\| E_{v}\|_{L^2(\Omega)}^{2}\right)(t^{{\ast}}) \\
  &=
  {\frac{1}{2}} \left(|E_{u}|^{2}_{H^1(\Omega)}+\| E_{v}\|_{L^2(\Omega)}^{2}\right)(0)
  \notag\\
  &+ \int_{0}^{t^{{\ast}}}\left(\cfrac{{\partial\tilde{v}_{{\tau} h}}}{\partial t}-f, E_{v}-\tilde I_{h} E_{v}\right) d t
  +\int_{0}^{t^{{\ast}}}\left({\nabla} \tilde{u}_{{\tau} h},{\nabla} ( E_{v}-\tilde I_{h}  E_{v})\right) d t
 \notag \\
  &+\int_{t_1}^{t^{{\ast}}}\left[\left(p_n \left(A_h {\partial}_{n}^{2} u_{h}-{\partial}_{n}^{2}f_h\right),\tilde I_{h} E_{v}\right)-\left(p_n {\nabla}{\partial}_{n}^{2} v_{h},{\nabla}E_{u}\right)
  +\left(\tilde{f}_\tau-f,\tilde I_{h} E_{v}\right)\right] d t \notag\\
  &+\int_{0}^{t_1}\left[\left(p_1 \left(A_h {\partial}_{1}^{2} u_{h}-{\partial}_{1}^{2}f_h\right),\tilde I_{h} E_{v}\right)-\left(p_1 {\nabla}{\partial}_{1}^{2} v_{h},{\nabla}E_{u}\right)
  +\left(\tilde{f}_\tau-f,\tilde I_{h} E_{v}\right)\right] d t 
  \notag\\
  &\hspace{1cm}:=I+I I+I I I+IV .\notag
  \\
    \label{4newterms} 
  &
\end{align}
Let
\begin{equation*}
 Z (t) = \sqrt{| E_u |_{H^1(\Omega)}^2 + \| E_v \|_{L^2(\Omega)}^2}
\end{equation*}
and assume that $t^{\ast}$ is the point in time where $Z$ attains its maximum and $t^{\ast} \in (t_n, t_{n + 1}]$
for some $n$. Observe
\begin{equation*}
 (I-\tilde I_h)E_v  = (I-\tilde I_h)(\tilde{v}_{h\tau}-v )
  = (I-\tilde I_h)\left(\cfrac{\partial\tilde{u}_{h\tau}}{\partial t}-\cfrac{\partial u}{\partial t} \right)
  = \cfrac{\partial}{\partial t}(I-\tilde I_h)E_{u}
\end{equation*}
since $( I - \tilde I_h) \varphi_h =0$ for any $\varphi_h\in V_h$.
We thus get for the first and second terms in (\ref{4newterms})
\begin{align*}
  I + I I &
  =\int_{0}^{t^{{\ast}}}\left(\cfrac{{{\partial } \tilde{v}_{{\tau} h}}}{\partial t}-f,\cfrac{\partial}{\partial t} (E_{u}-\tilde I_{h} E_{u})\right) d t+\int_{0}^{t^{{\ast}}}\left({\nabla} \tilde{u}_{{\tau} h},\cfrac{\partial}{\partial t} {\nabla} (E_{u}-\tilde I_{h}  E_{u})\right) dt.
\end{align*}
We now integrate by parts with respect to time in the two integrals above.
Let us do it for the first term:
\begin{align*}
\nonumber & \int_{0}^{t^{{\ast}}}\left(\cfrac{{{\partial } \tilde{v}_{{\tau} h}}}{\partial t}-f,\cfrac{\partial}{\partial t} (E_{u}-\tilde I_{h} E_{u})\right) d t \\
\nonumber & =
   \sum_{m = 0}^n \int_{t_m}^{\min (t_{m + 1}, t^{\ast})}\left(\cfrac{{{\partial } \tilde{v}_{{\tau} h}}}{\partial t}-f,\cfrac{\partial}{\partial t} (E_{u}-\tilde I_{h} E_{u})\right) d t \\
 \nonumber & = \left( \cfrac{{{\partial } \tilde{v}_{{\tau} h}}}{\partial t}-f, E_u -\tilde  I_h E_u \right)
   (t^{\ast}) - \sum_{m = 1}^n \left( \left[ \cfrac{{{\partial } \tilde{v}_{{\tau} h}}}{\partial t} \right]_{t_m}, (  E_u - \tilde I_h E_u) (t_n) \right) \\
 & - \sum_{m = 0}^n \int_{t_m}^{\min (t_{m + 1}, t^{\ast})} \left(
   \cfrac{{{\partial^2 } \tilde{v}_{{\tau} h}}}{\partial t^2}-\cfrac{\partial {f}}{\partial t},  E_u -\tilde  I_h E_u
   \right) dt .
\end{align*}
Here $[\cdot]_{t_n}$ denotes the jump with respect to time, i.e. 
$$[w]_{t_n} = \lim_{t
\rightarrow t_n^+} w ( t) - \lim_{t \rightarrow t_n^-} w ( t).$$
 Using the
same trick in the other term we can finally write
\begin{align}
  \nonumber I +I I &
  =\left( \cfrac{{{\partial } \tilde{v}_{{\tau} h}}}{\partial t}-f, E_u - \tilde I_h E_u \right)
   (t^{\ast})+\left({\nabla} \tilde{u}_{{\tau} h},{\nabla} ( E_{u}-\tilde I_{h} E_{u})\right)(t^{{\ast}})\\
  \nonumber &  -\sum_{m = 1}^n \left( \left[ \cfrac{{{\partial } \tilde{v}_{{\tau} h}}}{\partial t} \right]_{t_m}, (  E_u - \tilde I_h E_u) (t_n) \right)\\
  \nonumber & - \sum_{m = 0}^n \int_{t_m}^{\min (t_{m + 1}, t^{\ast})} \left(
   \cfrac{{{\partial^2 } \tilde{v}_{{\tau} h}}}{\partial t^2}-\cfrac{\partial {f}}{\partial t},  E_u -\tilde  I_h E_u
   \right) dt\\&-\sum_{m=0}^{n}\int_{t_{m}}^{min
  (t_{m+1},t^{{\ast}})}\left({\nabla} \cfrac{{{\partial } \tilde{u}_{{\tau} h}}}{\partial t},{\nabla} ( E_{u}-\tilde I_{h} E_{u})\right) d t.
\end{align}
We have used here a simple expression for the jump of time of
${\partial
\tilde{v}_{h\tau}}/\partial t$
\begin{equation}\left[ \cfrac{\partial \tilde{v}_{h\tau}}{\partial t} \right]_{t_n} = {\tau_{n
   - 1}}{2}  (\partial_n^2 v_h - \partial_{n - 1}^2 v_h) 
\end{equation}
and noted that $\tilde{u}_{h\tau}$ is continuous in time.

  Integration by parts element by element over $\Omega$ and interpolation estimates (\ref{Clement}) yield
\begin{align*}
  I+II
  &\leq C_{1}\Biggl[\sum_{K{\in}\mathcal{T}_h}h_{K}^{2}\left\Vert\cfrac{{\partial}\tilde{v}_{h{\tau}}}{\partial t}-{\Delta}\tilde{u}_{h{\tau}}-f\right\Vert_{L^2( K)}^{2}\\
 &~~~~~~~~~~~~~~~~~~~~~~~~~~~~+\sum_{
E \in \mathcal{E}_h}h_{E}\left\|[n{\cdot}{\nabla}\tilde{u}_{h{\tau}}]\right\|_{L^2(E)}^{2}\Biggr]^{1/2}(t^{{\ast}})|E_{u}|_{H^1(\Omega)}(t^{{\ast}}) \\
  &+C_{1}\Biggl[\sum_{K{\in}\mathcal{T}_h}h_{K}^{2}\left\Vert\cfrac{{\partial}\tilde{v}_{h{\tau}}}{\partial t}-{\Delta}\tilde{u}_{h{\tau}}-f\right\Vert_{L^2(K)}^{2}\\
  &~~~~~~~~~~~~~~~~~~~~~~~~~~~~~~+\sum_{
E \in \mathcal{E}_h}h_{E}\left\|[n{\cdot}{\nabla}\tilde{u}_{h{\tau}}]\right\|_{L^2(E)}^{2}\Biggl]^{1/2}(0)|E_{u}|_{H^1(\Omega)}(0)\\
  &
  +C_{2}\sum_{m=1}^{n}\frac{{\tau}_{m-1}}{2}\left[\sum_{K{\in}\mathcal{T}_h}h_{K}^{2}\left\|{\partial}_{m}^{2}v_{h}-{\partial}_{m-1}^{2}v_{h}\right\|_{L^2( K)}^{2}\right]^{1/2}|E_{u}|_{H^1(\Omega)}(t_{m})\\
  &
  +C_{3}\sum_{m=0}^{n}\int_{t_{m}}^{min(t_{m+1},t^{{\ast}})}\Biggl[\sum_{K{\in}\mathcal{T}_h}h_{K}^{2}\left\Vert\cfrac{{\partial}^{2}\tilde{v}_{h{\tau}}}{\partial t^2}-{\Delta}\cfrac{{\partial}\tilde{u}_{{\tau}h}}{\partial t}-\cfrac{\partial f}{\partial t}\right\Vert_{L^2(K)}^{2}\\
  &~~~~~~~~~~~~~~~~~~~~~~~~~~~~+\sum_{
E \in \mathcal{E}_h}h_{E}\left\Vert\left[n{\cdot}{\nabla}\cfrac{{\partial}\tilde{u}_{{\tau}h}}{\partial t}\right]\right\Vert_{L^2(E)}^{2}\Biggr]^{1/2 }(t)|E_{u}|_{H^1(\Omega)}(t)dt.
\end{align*}
We turn now to the third term in (\ref{4newterms})
\begin{align*}
 I I I &= \int_{t_1}^{t^{\ast}} \{(p_n (A_h {\partial}_{n}^{2} u_{h}-{\partial}_{n}^{2}f_h),\tilde I_{h} E_{v})-\left(p_n {\nabla}{\partial}_{n}^{2} v_{h},{\nabla}E_{u}\right)
  +(\tilde{f}_\tau-f,\tilde I_{h} E_{v}) \}dt \\
   & \leq C
   \sum_{m = 1}^n \Biggl[ \left( \int_{t_m}^{t^{}_{m + 1}} |p_m|{dt} \right) 
   \left( \left\Vert \partial_m^2 f_h - A_h \partial_m^2 u_h \right\Vert_{L^2(\Omega)} + \left\vert \partial_m^2 v_h\right\vert_{H^1(\Omega)}\right)\\
  &+\int_{t_m}^{t^{}_{m + 1}} \left\Vert f-\tilde{f}_\tau \right\Vert_{L^2(\Omega)} {dt} \Biggr] Z ( t^{^{\ast}}) \\
\end{align*}
with
$$
\int_{t_m}^{t_{m + 1}} |p_m|{dt} \leq \frac{1}{12}{\tau}_{m}^{3}+\frac{1}{8}{\tau}_{m-1}{\tau}_{m}^2.
$$
We have used here the bounds $|E_u |_{H^1(\Omega)} (t) \leqslant Z (t) \leqslant Z
(t^{\ast})$ and $\|E_v \|_{L^2(\Omega)} \leqslant Z (t) \leqslant Z (t^{\ast})$ for all $t
\in [0, t^{\ast}]$. Similar reasoning for the fourth term in (\ref{4newterms}) give us
\begin{align*}
 I V &= \int_{t_0}^{t_1} \{(p_1 (A_h {\partial}_{1}^{2} u_{h}-{\partial}_{1}^{2}f_h),\tilde I_{h} E_{v})-\left(p_1 {\nabla}{\partial}_{1}^{2} v_{h},{\nabla}E_{u}\right)
  +(\tilde{f}_\tau-f,\tilde I_{h} E_{v}) \}dt \\
   & \leq C
    \Biggl[ \left( \int_{t_0}^{t^{}_{ 1}} |p_1|{dt} \right) 
   \left( \left\Vert \partial_1^2 f_h - A_h \partial_1^2 u_h \right\Vert_{L^2(\Omega)} + \left\vert \partial_1^2 v_h\right\vert_{H^1(\Omega)}\right)\\
   &+\int_{t_0}^{t^{}_{1}} \left\Vert f-\tilde{f}_\tau \right\Vert_{L^2(\Omega)} {dt} \Biggr] Z ( t^{^{\ast}}) \\
\end{align*}
where
$$
\int_{t_0}^{t_{1}} |p_1|{dt} \leq \frac{5}{12}{\tau}_{0}^{3}+\frac{1}{2}{\tau}_{1}{\tau}_{0}^2.
$$
Applying the same bounds for $|E_u |_{H^1(\Omega)} (t)$ and $\|E_v \|_{L^2(\Omega)} \leqslant Z (t)$ to the estimates for integrals  $I+II$, inserting them into  (\ref{4newterms}) and noting that $A_h\partial^2_k u_h=z^k_h$ we obtain (\ref{estf}).
\end{proof}

\begin{remark}
 Comparing the a priori estimate (\ref{apriori}) with the a posteriori one (\ref{estf}) one sees that the time error indicator is essentially the same in both cases. Indeed, the term $\displaystyle\int_{t_n}^{t_{n + 1}} \left\Vert \cfrac{\partial^4 u }{\partial t^4} \right\Vert_{L^2(\Omega)} {dt}$ can be rewritten as $\displaystyle\int_{t_n}^{t_{n + 1}} \left\Vert \cfrac{\partial^2 f }{\partial t^2} + \Delta \cfrac{\partial^2 u }{\partial t^2} \right\Vert_{L^2(\Omega)} {dt}$ and it's discrete counterpart is in \ref{in} and \ref{in0step}. 
 Note also that the last term in (\ref{estf}) is negligible, at least if $f$ the sufficiently smooth in time, since $\|f-\tilde{f}_\tau\|_{L^2(\Omega)}=O(\tau_n^3)$ for $t\in(t_n,t_{n+1})$.
 
Moreover, in view of a posteriori estimate some of the terms are of higher order $\tau h^2$, so that neglecting the higher order terms, a posteriori space error estimator can be reduced to the two first lines in (\ref{space}), i.e.
 \begin{align}
 \label{space1}  \eta_S^{(1)} (t_k) &= C_1 \max_{0 \leqslant t \leqslant t_k} \Biggl[ \sum_{K \in \mathcal{T}_h}
   h_K^2  \left\Vert \frac{\partial \tilde{v}_{h \tau}}{\partial t} - \Delta \tilde{u}_{h \tau}-f
   \right\Vert_{L^2(K)}^2\\
   \notag&~~~~~~~~~~~~~~~~~~~~~~~~~~~~~~~~~~~~~~+ \sum_{
E \in \mathcal{E}_h}h_{E} \|[n \cdot \nabla \tilde{u}_{h
   \tau}]\|_{L^2(E)}^2 \Biggr]^{1/2} (t),\\
\label{space2}
 \eta_S^{(2)} (t_k)&= C_2\sum_{m = 0}^k \int_{t_m}^{t_{m + 1}} \Biggl[ \sum_{K \in \mathcal{T}_h} h_K^2
  \left\Vert \frac{\partial^2 \tilde{v}_{h \tau}}{\partial t^2} - \Delta \frac{\partial
  \tilde{u}_{h\tau}}{\partial t} -\frac{\partial{f}}{\partial{t}}\right\Vert_{L^2(K)}^2\\
  \notag&~~~~~~~~~~~~~~~~~~~~~~~~~~~~~~~~+ \sum_{
E \in \mathcal{E}_h}h_{E} \left\Vert\left[n \cdot
  \nabla \frac{\partial \tilde{u}_{h\tau}}{\partial t}\right]\right\Vert_{L^2(E)}^2
  \Biggr]^{1/2} (t) dt
 \end{align}
\end{remark}

\subsection{Optimality of the error estimators}\label{optimality}
We do not have a lower bound for our error estimators in space and time. Note that such a bound is not available even in a simpler setting of Euler discretization in time, cf. \cite{BS}. We are going to prove a partial result in the direction of optimality, namely that the indicator of error in time provides the estimate of order $\tau^2$ at least on sufficiently smooth solutions and quasi-uniform meshes. For this, we should examine if the quantities $\partial_n^2 f_h - A_h \partial_n^2 u_h$ and $\partial^2_n v_h$ remain bounded in $L^2$ and $H^1$ norms respectively.  This will be achieved in Lemma \ref{bound_d4u} assuming that the initial conditions are discretized in a specific way, via the $H^1_0$-orthogonal projection.

We restrict ourselves to the constant time steps $\tau_n=\tau$ and introduce the notations
  \begin{align*}
    {\partial}^{0}_{n} u_h&=u_h^{n+1},&&\quad
    {\partial}^{j+1}_{n} u_h={\frac{{\partial}^{j}_{n}u_h-{\partial}^{j}_{n-1}u}{{\tau}}},
    &&\quad
    j=0,1,{\ldots},\quad n\ge j-1\\
    \bar{\partial}^0_n {u}_h&=\frac{u_h^{n + 1} + u_h^n}{2},&&\quad\bar{\partial}^{j + 1}_n  {u}_h= \frac{\bar{\partial}^j_n
       {u}_h - \bar{\partial}^j_{n - 1} {u}_h}{\tau},     
    &&\quad
    j=0,1,{\ldots},\quad n\ge j
  \end{align*}
The Crank-Nicolson scheme for first-order system (\ref{CNh1})-(\ref{CNh2}) for $n\ge 0$ is written with these notations as  
\begin{align}
   \label{CNh1N}
   {\partial}^{1}_{n} u_h - \bar{\partial}^0_n {v}_h &= 0\\
   \label{CNh2N}
   \left( {\partial}^{1}_{n} v_h, \varphi_h \right) + \left( \nabla \bar{\partial}^0_n {u}_h, \nabla \varphi_h \right) &= \left(  \bar{\partial}^0_n {f}_h, \varphi_h \right), \hspace{1em} \forall \varphi_h \in V_h
\end{align}
where $f^n_h$, $n\ge 0$, are the $L^2$-orthogonal projection of $f(t_n,\cdot)$ on $V_h$. The following lemma provides a higher regularity result on the discrete level, i.e. the boundedness of terms $\partial_n^j f_h - A_h \partial_n^j u_h$ and $\partial^j_n v_h$ for any $j\in\mathbb{N}^0$.

\begin{lemma}\label{bound_d4u_abst} Let $u_h^n$ and $v_h^n$ be the solution to (\ref{CNh1})-(\ref{CNh2}) for $n\ge 0$. One has then for all $j\in\mathbb{N}^0$, $N\in\mathbb{N}$, $N\ge j$
\begin{multline}\label{boundj}
 \left( \left\| \partial^{j}_N f_h-A_h\partial^{j}_N u_h \right\|_{L^2(\Omega)}^2 + \left| {\partial}^j_N v_h \right|_{H^1(\Omega)}^2\right)^{{1}/{2}} \\
 \leq \left( \left\| \partial^{j}_jf_h-A_h\partial^{j}_j u_h \right\|_{L^2(\Omega)}^2 + \left|{\partial}^j_j v_h \right|_{H^1(\Omega)}^2\right)^{{1}/{2}} + \tau\sum_{n = j+1}^N\left\| {\partial}^{j+1}_n {f}\right\|_{L^2(\Omega)}
\end{multline}
\end{lemma}

\begin{proof}
Starting from (\ref{CNh1N})-(\ref{CNh2N}), taking the differences between steps $n$ and $n - 1$
and then making an induction on $j = 0, 1, \ldots$ one arrives at
\begin{align}
   \label{CNh1NN}
   {\partial}^{j+1}_{n} u_h &= \bar{\partial}^j_n {v}_h, \\
   \label{CNh2NN}
    {\partial}^{j+1}_{n} v_h  &=   \bar{\partial}^j_n {f}_h - A_h \bar{\partial}^j_n {u}_h. 
\end{align}
One can also prove that  $\forall w^n_h\in V_h$ 
\begin{equation}\label{propj}
 \bar{\partial}^j_n {w}_h = \frac{{\partial}^j_n w_h +
   {\partial}^j_{n - 1} w_h}{2}, \hspace{1em} j = 0, 1, \ldots
\end{equation}
Indeed, this is obvious for $j = 0$ and then it follows for any $j$ by induction.

Taking the inner product of (\ref{CNh2NN}) with $\tau A_h{\partial}^{j + 1}_n u_h-\tau {\partial}^{j + 1}_n f_h$, using (\ref{propj}) and definition of ${\partial}^{j+1}_n$  we obtain
\begin{align*}
 \biggl(\partial_n^{j+1} v_h ,
   \tau A_h{\partial}^{j + 1}_n u_h&-\tau {\partial}^{j + 1}_n f_h \biggr)=\biggl(\bar{\partial}_n^{j} f_h - A_h\bar{\partial}_n^{j} u_h, \tau A_h{\partial}^{j + 1}_n u_h-\tau {\partial}^{j + 1}_n f_h \biggr)\\
&= -\frac{\left\| \partial^{j}_n f_h-A_h\partial^{j}_n u_h  \right\|_{L^2(\Omega)}^2}{2}
   + \frac{\left\|  \partial^{j}_{n-1} f_h-A_h\partial^{j}_{n-1} u_h  \right\|_{L^2(\Omega)}^2}{2}.
\end{align*}
Now we apply (\ref{propj}) and (\ref{CNh1NN}) to the left-hand side above 
\begin{align*}
 \biggl(\partial_n^{j+1} v_h ,
   \tau A_h{\partial}^{j + 1}_n u_h&-\tau {\partial}^{j + 1}_n f_h \biggr)
   \\
    &=\left({\partial_n^{j } v_h - \partial_{n-1}^{j} v_h} , A_h{\partial}^{j + 1}_n u_h\right) 
   -\left(\partial_n^{j+1 } v_h , \tau{\partial}^{j + 1}_n f_h\right)\\
   &=  \frac{\left| {\partial}^j_n v_h \right|^2_{H^1(\Omega)} - \left| {\partial}^j_{n - 1}
   v_h \right|_{H^1(\Omega)}^2}{2}
   -\left(\partial_n^{j+1 } v_h , \tau{\partial}^{j + 1}_n f_h\right).
\end{align*}
Thus \begin{align*}
  \frac{\left| {\partial}^j_n v_h \right|^2_{H^1(\Omega)} - \left| {\partial}^j_{n - 1}
   v_h \right|_{H^1(\Omega)}^2}{2}
   -\left(\partial_n^{j+1 } v_h , \tau{\partial}^{j + 1}_n f_h\right)&= -\frac{\left\| \partial^{j}_n f_h-A_h\partial^{j}_n u_h  \right\|_{L^2(\Omega)}^2}{2}\\
   &+ \frac{\left\|  \partial^{j}_{n-1} f_h-A_h\partial^{j}_{n-1} u_h  \right\|_{L^2(\Omega)}^2}{2}.
\end{align*}
We recall by (\ref{CNh2NN}) 
\begin{align*}
\tau{\partial}^{j+1}_{n}v_h&=\tau\left( \bar{\partial}^{j}_{n}f_h-A_h\bar{\partial}^{j}_{n} u_h \right)\\
&= \frac{\tau}{2}\left({\partial}^{j}_{n}f_h+{\partial}^{j}_{n-1}f_h-A_h{\partial}^{j}_{n-1} u_h - A_h{\partial}^{j}_{n-1} u_h \right)
\end{align*}
and hence 
\begin{align*}
\left\| \partial^{j}_nf_h-A_h\partial^{j}_n u_h \right\|_{L^2(\Omega)}^2 + \left| {\partial}^j_{n} v_h \right|_{H^1(\Omega)}^2-\left\| \partial^{j}_{n-1}f_h-A_h\partial^{j}_{n-1} u_h \right\|_{L^2(\Omega)}^2 - \left| {\partial}^j_{n-1} v_h \right|_{H^1(\Omega)}^2 \\
\leq \tau\left\| {\partial}^{j+1}_n {f}_h \right\|_{L^2(\Omega)}\left(\left\| \partial^{j}_nf_h-A_h\partial^{j}_n u_h \right\|_{L^2(\Omega)}+ \left\| \partial^{j}_{n-1}f_h-A_h\partial^{j}_{n-1} u_h \right\|_{L^2(\Omega)}\right).
\end{align*}
Denoting $Z_n=\left(\left\| \partial^{j}_nf_h-A_h\partial^{j}_n u_h \right\|_{L^2(\Omega)}^2 + \left| {\partial}^j_{n} v_h \right|_{H^1(\Omega)}^2\right)^{{1}/{2}}$ the last inequality can be rewritten as
\begin{align*}
Z_n^2-Z_{n-1}^2
&\leq \tau\left\| {\partial}^{j+1}_n {f}_h \right\|_{L^2(\Omega)}\left(\left\| \partial^{j}_nf_h-A_h\partial^{j}_n u_h \right\|_{L^2(\Omega)} \right.\\
& \left.+\left\| \partial^{j}_{n-1}f_h-A_h\partial^{j}_{n-1} u_h \right\|_{L^2(\Omega)}\right)
\leq \tau\left\| {\partial}^{j+1}_n {f}_h \right\|_{L^2(\Omega)} (Z_n+Z_{n-1})
\end{align*}
 so that
\[
Z_n-Z_{n-1} \leq \tau\left\| {\partial}^{j+1}_n {f}_h \right\|_{L^2(\Omega)}. 
\]
Summing this over $n$ we get (\ref{boundj}).
 \end{proof}
 
In order to take into account the initial conditions, we shall need the following auxiliary result about stability properties of operator $A_h$ defined by (\ref{Ah}) and the $L^2$-orthogonal projection  $P_h : L^2 (\Omega) \to V_h$ defined by
\begin{equation}\label{Ph}
 \forall v \in L^2 (\Omega) : \left( P_h v, \varphi_h\right) = \left( v,
   \varphi_h\right) \hspace{1em} \forall \varphi_h \in V_h 
\end{equation}

\begin{lemma} \label{bound_AhPhu}
Assuming the mesh $\mathcal{T}_h$ to be quasi-uniform, there exists  $C>0$ depending only on the regularity of $\mathcal{T}_h$ such that
\begin{eqnarray}\label{Phbound}
\forall v \in H^1_0 (\Omega) & :  & | P_h v |_{H^1(\Omega)} \leq C | v |_{H^1(\Omega)},
\\
\label{AhPhbound}
\forall v \in H^2 (\Omega) \cap H^1_0 (\Omega) & : & \| A_hP_h v \|_{L^2(\Omega)} \leq C | v |_{H^2(\Omega)}
\end{eqnarray} 

\end{lemma}
\begin{proof}
Let $ v \in H^1_0 (\Omega)$. 
Using a Cl{\'e}ment-type interpolation operator $\tilde I_h$, satisfying $\tilde I_h=Id$ on $V_h$ and (\ref{Clement}), together with an inverse inequality we observe 
\[
| P_h v |_{H^1(\Omega)} 
\leq | P_h v  -\tilde I_hv  |_{H^1(\Omega)} + | \tilde I_h v  |_{H^1(\Omega)} \leq \cfrac{C}{h}\| P_h v  -\tilde I_hv  \|_{L^2(\Omega)} + | v |_{H^1(\Omega)}
\]
Then, from approximation properties (\ref{Clement}) 
\[\| P_h v  -v  \|_{L^2(\Omega)} \leq \| \tilde I_h v  -v  \|_{L^2(\Omega)}  \leq C h| v  |_{H^1(\Omega)}
 \leq C h| v |_{H^1(\Omega)}
 \]
which entails (\ref{Phbound}).

We assume now $v\in H^2(\Omega) \cap H^1_0 (\Omega)$ and use a similar idea to prove (\ref{AhPhbound}). For any $\varphi_h\in V_h$
\begin{equation}\label{tr}
\left( A_hP_h v,\varphi_h\right)
= \left(\nabla\left(P_h-\tilde I_h\right) v ,\nabla\varphi_h\right)+\left( \nabla \tilde I_h v ,\nabla\varphi_h\right)
\end{equation}
We can bound the first term in the right-hand side of (\ref{tr}) using the inverse inequality and the approximation properties of $\tilde I_h$  
\[\left(\nabla\left(P_h-\tilde I_h\right) v ,\nabla\varphi_h\right)\leq \cfrac{C}{h^2}\| P_h v  -\tilde I_hv  \|_{L^2(\Omega)}\| \varphi_h \|_{L^2(\Omega)}\leq C| v  |_{H^2(\Omega)}\| \varphi_h \|_{L^2(\Omega)}
\]
To deal with the second term in the right-hand side of (\ref{tr}), we integrate by parts over all the triangles of the mesh and recall that $\Delta\varphi_h=0$ on any triangle, so that
\[
\left( \nabla \tilde I_h v ,\nabla\varphi_h \right)=\sum_{E \in \mathcal{E}_h}\int_{E}\left[\cfrac{\partial \tilde I_h v }{\partial n}\right]\varphi_h
\leq \sum_{E \in \mathcal{E}_h}\left\Vert\left[\cfrac{\partial \tilde I_h v }{\partial n}\right]\right\Vert_{L^2(E)}\left\Vert\varphi_h\right\Vert_{L^2(E)}
\]
Using the inverse trace inequality $\left\Vert\varphi_h\right\Vert_{L^2(E)}\leq \cfrac{C}{\sqrt{h}}\left\Vert\varphi_h\right\Vert_{L^2(\omega_E)}$ and the interpolation error bound  
$$
\left\Vert\left[\cfrac{\partial \tilde I_h v }{\partial n}\right]\right\Vert_{L^2(E)}
=\left\Vert\left[\cfrac{\partial  }{\partial n}(v-\tilde I_h v)\right]\right\Vert_{L^2(E)}
\leq C\sqrt{h}| v  |_{H^2(\omega_{E})}
$$ on all the edges $E\in \mathcal{E}_h$ leads, together with (\ref{tr}), to
$$
\left( A_hP_h v,\varphi_h\right) 
\le 
C| v  |_{H^2(\Omega)}\| \varphi_h \|_{L^2(\Omega)}
$$
Taking here $\varphi_h=A_hP_h v$, we obtain desired result (\ref{AhPhbound}).
\end{proof}

\begin{remark}
Our proof of Lemma \ref{bound_AhPhu} uses  inverse inequalities and is thus restricted to the quasi-uniform meshes $\mathcal{T}_h$. The first estimate (\ref{Phbound}) is actually established in \cite{bramble2002stability} under much milder hypotheses on the mesh compatible with usual mesh refinement techniques. We conjecture that the second estimate (\ref{AhPhbound}) also holds under similar assumptions. Some numerical examples in this direction are given at the end of Subsection \ref{unstr}.
\end{remark}
 
We are now able to complete the estimate of Lemma \ref{bound_d4u_abst} in the case $j=2$ which is pertinent to our a posteriori analysis.

\begin{lemma}\label{bound_d4u} Let $u_h^n$ be the solution to (\ref{Newm1})--(\ref{Newm2}) on a quasi-uniform mesh with
\begin{equation}
u^0_h = \Pi_h u^0,~v^0_h = \Pi_h v^0 \label{initialcond}
\end{equation}
where $\Pi_h$ is the $H^1_0$-orthogonal projection on $V_h$.
 One has for all $N\ge 1$
\begin{align}
\label{boundj4} \biggl(\bigl\| \partial^{2}_Nf_h-A_h&\partial^{2}_N u_h \bigr\|_{L^2(\Omega)}^2  + \left| {\partial}^2_{N} v_h \right|_{H^1(\Omega)}^2\biggr)^{{1}/{2}}\\
\notag &\leq C \left(\left| \cfrac{\partial^{3}u}{\partial t^{3}}(0) \right|_{H^1(\Omega) }  + \left| \cfrac{\partial^{2}u}{\partial t^{2}}(0) \right|_{H^2(\Omega)} + \max_{t \in [0,2\tau]} \left\| \cfrac{\partial^{2}f}{\partial t^{2}}(t)\right\|_{L^2(\Omega)}
  \right)\\
\notag&+ \int_0^{t_N} \left\| \cfrac{\partial^{3}f}{\partial t^{3}}\right\|_{L^{2}(\Omega)}dt 
\end{align}
with a constant $C > 0$ independent of $h$, $\tau$, $N$.
\end{lemma}

\begin{proof} Denote
\[ Z = 2 \left( I + \frac{\tau^2}{4} A_h \right)^{- 1} \left( I -
   \frac{\tau^2}{4} A_h \right) \]
Then scheme (\ref{Newm2}) for $n \geq 1$ can be rewritten as
\[ u^{n + 1}_h = Zu_h^n - u^{n - 1}_h + \tau^2 \left( I + \frac{\tau^2}{4} A_h
   \right)^{- 1} \bar{f}^n_h \]
Moreover, the initial step (\ref{Newm1}) can be written as
\begin{equation*}
  \frac{u^1_h - u^0_h - \tau v^0_h}{\tau^2} + A_h ^{} \frac{u^1_h + u^0_h}{4}
  = \bar{f}^0_h := \frac{f^1_h + f^0_h}{4} \label{Newmh0}
\end{equation*}
This gives the following expressions for $u_h^1, u_h^2$:
\begin{align*}
\notag u_h^1 &= \tau^2 \left( I + \frac{\tau^2}{4} A_h \right)^{- 1} \left(
   \bar{f}^0_h + \frac{1}{\tau} v^0_h \right) + \frac{1}{2} Zu^0_h \\
\notag u_h^2 &= \tau^2 \left( I + \frac{\tau^2}{4} A_h \right)^{- 1} \left( Z
   \left( \bar{f}^0_h + \frac{1}{\tau} v^0_h \right) + \bar{f}^1_h \right) +
   \left( \frac{1}{2} Z^2_{} - I \right) u_h^0 
\end{align*}
Thus,
\begin{align*}
 \partial^{2}_1f_h-A_h\partial^{2}_1 u_h &=   \partial^{2}_1f_h - \frac{A^2_hZ}{2\left(I + \frac{\tau^2}{4} A_h \right)}u^0_h \\
 &- {A_h}\left( I + \frac{\tau^2}{4} A_h \right)^{- 1} \left( (Z-2I)
   \left( \bar{f}^0_h + \frac{1}{\tau} v^0_h \right) + \bar{f}^1_h \right)
\end{align*}
and
\begin{align*}
 {\partial}^2_1 v_h = &-A_h\frac{u^2_h-u^0_h}{2\tau}+\frac{f^2_h-f^0_h}{2\tau}=-\frac{A_h}{2 \tau} \left(\frac{1}{2}Z^2 - 2 I\right) u^0_h \\
   &-   \frac{A_h}{2 \tau} \tau^2 \left( I + \frac{\tau^2}{4} A_h \right)^{- 1} \left( Z
   \left( \bar{f}^0_h + \frac{1}{\tau} v^0_h \right) + \bar{f}^1_h \right) +\frac{f^2_h-f^0_h}{2\tau} 
\end{align*} 
After some tedious calculations,  this can be rewritten as
\begin{align}
\notag \partial^{2}_1f_h-A_h\partial^{2}_1 u_h  =  -\frac{1}{2}  \frac{Z} {\left( I +
  \frac{\tau^2}{4} A_h \right)^2}  \left(A_h^2 u^0_h - A_h f^0_h\right) &+\frac{\tau A_h
  }{\left( I + \frac{\tau^2}{4} A_h \right)^2}  \left(A_h v^0_h - \partial^1_0 f_h\right)\\
 &+\left( I + \frac{\tau^2}{4} A_h \right)^{- 1}\partial^2_1 {f}_h 
  \label{expAhd2u} 
  \end{align}
and
\begin{align}
  \label{expd2v}{\partial}^2_1 v_h = -\frac{\tau}{\left( I +
  \frac{\tau^2}{4} A_h \right)^2}  \left(A_h^2 u^0_h - A_h f^0_h\right) &+ \frac{Z}{2 \left( I + \frac{\tau^2}{4} A_h \right)}  \left(A_h v^0_h - \partial^1_0
  f_h\right) \\
 \notag &- \frac{\tau}{2 \left( I + \frac{\tau^2}{4} A_h \right)}   \partial^2_1 f_h 
\end{align}
Since $A_h$ is a symmetric positive definite operator, we have
\[ \| R (\tau^2 A_h) v_h \|_{L^2(\Omega)} \leq C \| v_h \|_{L^2(\Omega)} \]
for any $v_h \in V_h$ and any rational function $R$ with the degree of nominator
less or equal than that of the denominator and a constant $C$ depending only
on $R$. 
Similarly, using the fact $| v_h
|_{H^1(\Omega)} = (A_hv_h,v_h)^\frac{1}{2}=\left\|A_h^{{1}/{2}} v_h\right\|_{L^2(\Omega)}$ for any $v_h\in V_h$ one can observe
\[ \| \tau A_h R (\tau^2 A_h) v_h \|_{L^2(\Omega)} 
\le C \| A_h^{1/2} v_h \|_{L^2(\Omega)} 
= C | v_h |_{H^1(\Omega)} \]
for any rational function $R$ with the degree of nominator
less than that of the denominator and a constant $C$ depending only on $R$. 

Applying these estimates to (\ref{expd2v}) yields
  \begin{align*}
    \| \partial^2_1 f_h - A_h \partial^2_1 u_h \|_{L^2 (\Omega)} & \leq C
    \left( \| A_h^2 u^0_h - A_h f^0_h \|_{L^2 (\Omega)} + \left| A_h v^0_h -
    \frac{\partial f_h}{\partial t} (0) \right|_{H^1 (\Omega)} \right.\\
    & \left. + \left\| \frac{\tau A_h}{\left( I + \frac{\tau^2}{4} A_h
    \right)^2}  \left( \frac{\partial f_h}{\partial t} (0) - \partial^1_0 f_h
    \right) \right\|_{L^2 (\Omega)} + \| \partial^2_1 f_h \|_{L^2 (\Omega)}
    \right)
  \end{align*}
Since
  \begin{equation*}
    \partial^1_0 f_h = \frac{\partial f_h}{\partial t} (0) +
    \frac{1}{\tau}  \int^{\tau}_0 (\tau - s)  \frac{\partial^2 f}{\partial
    t^2} (s) ds
  \end{equation*}
we have
\begin{align*}
  \left\| \frac{\tau A_h}{\left( I + \frac{\tau^2}{4} A_h \right)^2} 
     \left( \frac{\partial f_h}{\partial t} (0) - \partial^1_0 f_h \right)
     \right\|_{L^2 (\Omega)} &\leq \max_{t \in [0, \tau]} \left\| \frac{\tau^2
     A_h}{\left( I + \frac{\tau^2}{4} A_h \right)^2} \frac{\partial^2
     f_h}{\partial t^2} (t) \right\|_{L^2 (\Omega)} \\
     &\leq C\max_{t \in [0,
     \tau]} \left\| \frac{\partial^2 f_h}{\partial t^2} (t) \right\|_{L^2
     (\Omega)} 
\end{align*} 
Noting finally that $\| \partial^2_1 f_h \|_{L^2(\Omega)}$ can be bounded by the maximum of $\left\| \cfrac{\partial^{2}f}{\partial t^{2}}(t)\right\|_{L^2(\Omega)}$ over time interval $[0,2\tau]$, we arrive at 
\begin{align*}
   \|\partial^{2}_1f_h-A_h\partial^{2}_1 u_h\|_{L^2(\Omega)}  &\leq  C \left( \left\|A_h^2 u^0_h - A_h f^0_h
   \right\|_{L^2(\Omega)}+ \left|A_h v^0_h - \frac{\partial f_h}{\partial t}(0) \right|_{H^1(\Omega)}\right.\\
   &\left. + \max_{t\in[0,2\tau]}\left\|\frac{\partial^2 f}{\partial t^2}(t)\right\|_{L^2(\Omega)}\right) 
\end{align*}

By a similar reasoning we can also bound $\left| {\partial}^2_{1} v_h \right|_{H^1(\Omega)}$ by the same quantitity as in the right-hand side of the equation above. 
For this, we take the $H^1$ norm on both sides of (\ref{expd2v}) and observe for the first term on the right hand side
\begin{align*}
\left| 
\frac{\tau}{\left( I + \frac{\tau^2}{4} A_h \right)^2}  \left(A_h^2 u^0_h - A_h f^0_h\right)
\right|_{H^1(\Omega)}
&= \left\| 
\frac{\tau A_h^{1/2}}{\left( I + \frac{\tau^2}{4} A_h \right)^2}  \left(A_h^2 u^0_h - A_h f^0_h\right)
\right\|_{L^2(\Omega)}
\\
&\le C \left\| 
 A_h^2 u^0_h - A_h f^0_h
\right\|_{L^2(\Omega)}
\end{align*}
The other terms can be treated similarly so that, skipping some details, we obtain
\begin{multline}\label{boundedness}
 \left(\left\| \partial^{2}_1f_h-A_h\partial^{2}_1 u_h \right\|_{L^2(\Omega)}^2 + \left| {\partial}^2_{1} v_h \right|_{H^1(\Omega)}^2\right)^{{1}/{2}}
\leq C \left( \left\|A_h^2 u^0_h - A_h f^0_h
   \right\|_{L^2(\Omega)}\right.\\
   \left.+  \left|A_h v^0_h - \frac{\partial f_h}{\partial t}(0) \right|_{H^1(\Omega)} 
   + \max_{t\in[0,2\tau]}\left\|\frac{\partial^2 f}{\partial t^2}(t)\right\|_{L^2(\Omega)} \right) 
\end{multline}   

We can now invoke the estimate of Lemma \ref{bound_d4u_abst} with $j=2$ and combine it with (\ref{boundedness}). This gives 
\begin{multline}\label{boundj3}
  \left(\left\| \partial^{2}_Nf_h-A_h\partial^{2}_N u_h \right\|_{L^2(\Omega)}^2 + \left| {\partial}^2_{N} v_h \right|_{H^1(\Omega)}^2\right)^{{1}/{2}}\leq  \sum_{n = 3}^N \tau \left\| {\partial}^3_n f\right\|_{L^2(\Omega)}\\
  + C \left( \left\|A_h^2 u^0_h - A_h f^0_h \right\|_{L^2(\Omega)} + \left|A_h v^0_h - \frac{\partial f_h}{\partial t}(0) \right|_{H^1(\Omega)}\right.\\ + \max_{t\in[0,\tau]}\left\|\frac{\partial^2 f}{\partial t^2}(t)\right\|_{L^2(\Omega)}
\Bigg).
\end{multline}
The first term in the right-hand side in (\ref{boundj3}) can be easily bounded by $\displaystyle\int_0^{t_N} \left\| \cfrac{\partial^{3}f}{\partial t^{3}} \right\|_{L^2(\Omega)}dt$. The remaining terms in the middle line of (\ref{boundj3}) are bounded using Lemma \ref{bound_AhPhu} and the relation $A_h\Pi_h=-P_h\Delta$ as follows
$$ 
\left\|A_h^2 u^0_h - A_h f^0_h \right\|_{L^2 (\Omega)} = \left\|A_h P_h (- \Delta u^0 -
   f^0) \right\|_{L^2(\Omega) }
    = \left\|A_h P_h \cfrac{\partial^{2}u}{\partial t^{2}} (0)\right \|_{L^2(\Omega)
   }
   \leq C \left|  \cfrac{\partial^{2}u}{\partial t^{2}} (0) \right|_{H^2(\Omega) } 
$$
and
$$
\left|A_h v^0_h -\frac{\partial f_h}{\partial t}(0) \right|_{H^1(\Omega)
   } = \left|P_h  \left(- \Delta v^0 - \frac{\partial f}{\partial t}(0)\right)
   \right|_{H^1 (\Omega)} \\
   \leq \left|P_h \cfrac{\partial^{3}u}{\partial t^{3}} (0) \right|_{H^1(\Omega) }
   \leq C
   \left| \cfrac{\partial^{3}u}{\partial t^{3}} (0) \right|_{H^1(\Omega)} 
$$
This gives (\ref{boundj4}). 
\end{proof}

\begin{remark}
Note that in Lemma \ref{bound_d4u} the approximation of the initial conditions and of the right-hand side is crucial for boundedness of higher order discrete derivatives and consequently to optimality of our time and space error estimators. We illustrate this fact with some numerical examples in Subsection \ref{unstr}.  
\end{remark}
\begin{corollary}
Let u be the solution of wave equation (\ref{wave}) and 
$\displaystyle\cfrac{\partial^{3}u}{\partial t^{3}}(0)\in {H^1(\Omega)}$, $\displaystyle\cfrac{\partial^{2}u}{\partial t^{2}}(0) \in {H^2(\Omega)}$, $\displaystyle\cfrac{\partial^{2}f}{\partial t^{2}}(t)\in L^{\infty}(0,T;{L^2(\Omega)})$, $\displaystyle\cfrac{\partial^{3}f}{\partial t^{3}}(t)\in L^{2}(0,T;{L^2(\Omega)})$. Suppose that mesh $\Th$ is quasi-uniform and the mesh in time is uniform ($t_k=k\tau$). Then, the 3-point time error estimator $\eta_T(t_k)$ defined by (\ref{in},\ref{in0step}) is of order $\tau^2$, i.e. 
\begin{equation}
\eta_T(t_k)\leq C \tau^2.\label{upperOpt}
\end{equation}
with a positive constant $C$ depending only on $u$, $f$, and the mesh regularity.
\end{corollary}
\begin{proof}
Follows immediately from Lemma \ref{bound_d4u}. 
\end{proof}
\section{Numerical results}\label{section4}
\subsection{A toy model: a second order ordinary differential equation}

Let us consider first the following ordinary differential equation
\begin{equation}
\begin{cases}
\cfrac{d^{2}u(t)}{dt^{2}}+Au(t)=f(t) ,&t\in\left[ 0;T\right]\\
u(0)=u_0 ,&\\
\cfrac{du}{dt}(0)=v_0&
\end{cases}
\label{ODE}
\end{equation}
with a constant $A>0$. This problem serves as simplification of the wave equation in which we get rid of the space variable. The Newmark scheme reduces in this case to
\begin{align}
\frac{u^{n+1}-u^{n}}{\tau_n}-\frac{u^{n}-u^{n-1}}{\tau_{n-1}}&+A\frac{\tau_n(u^{n+1}+u^{n})+\tau_{n-1}(u^{n}+u^{n-1})}{4}=
\notag\\
&=\frac{\tau_n (f^{n+1}+f^n)+\tau_{n-1}(f^n+f^{n-1})}{4},~1\leq n\leq N-1
\label{schN}\\
\frac{u^1-u^0}{\tau_0}&=v_0-\frac{\tau_0}{4}A(u^1+u^0)+\frac{\tau_0}{4}(f^1+f^0),
\notag\\
u^0&=u_0\notag
\end{align}
the error becomes $e=\displaystyle\max_{0 \leq n \leq N}\left( \left\vert v^{n}-{u}'(t_{n})\right\vert ^{2}+A\left\vert u^{n}-u(t_{n})\right\vert ^{2}\right) ^{{1}/{2}}$, and the 3-point a posteriori error estimate $\forall n:~0\leq n \leq N$ simplifies to this form:
\begin{align}\label{errest}
e \leq \sum_{k=0}^{n-1}\tau_k\eta_{T}(t_k)&= \tau_0 \left(\frac{5}{12}\tau_{0}^2+\frac{1}{2}\tau_{0}\tau_{1}\right) \sqrt{A(\partial_1^2 v)^2 +  (\partial_1^2f-A\partial_1^2 u)^2}\\
\notag   &+\sum_{k=1}^{n-1}\tau_k \left(\frac{1}{12}\tau_{k}^2+\frac{1}{8}\tau_{k-1}\tau_{k}\right) \sqrt{A(\partial_k^2 v)^2 +  (\partial_k^2f-A\partial_k^2 u)^2}.
\end{align}

We define the following effectivity index  in order to measure the quality of our estimators $\eta_T$:
\begin{equation*}
ei_T=\frac{\eta_T}{e}.
\end{equation*}
We present in Table \ref{tab:ode} the results for equation (\ref{ODE}) setting $f=0$, the exact solution $u=cos(\sqrt{A}t)$, final time $T=1$, and using constant time steps $\tau=\displaystyle {T}/{N}$. We observe that 3-point estimator is divided by about 100 when the time step $\tau$ is divided by 10. The true error $e$ also behaves as $O(\tau^2)$ and hence the time error estimator behaves as the true error.

\begin{table}[t!]
\tblcaption{Effective indices for constant time steps and $f=0$.}
{
\begin{tabular}{@{}ccccc@{}}
\tblhead{$A$ & $N$ & $\eta_{T}$ & $e$ & $ei_{T}$ }
\noalign{\vskip 2mm} 
100\phzz & \phzz100 & 0.21~\phzzz & 0.085\phzzz & 2.47\\
100\phzz & \phz1000 & 0.0021\phzz & 8.34e-04    & 2.5\phz\\
100\phzz & 10000    & 2.08e-05    & 8.35e-06    & 2.5\phz\\
\noalign{\vskip 2mm} 
1000\phz & \phzz100 & 20.51\phzzz & 8.35~\phzzz & 2.46\\
1000\phz & \phz1000 & 0.209\phzzz & 0.084\phzzz & 2.5\phz\\
1000\phz & 10000    & 0.0021\phzz & 8.33e-04    & 2.5\phz\\
\noalign{\vskip 2mm} 
10000    & \phzz100 & 1.68e+03    & 200~~\phzzz & 8.38\\
10000    & \phz1000 & 20.8~\phzzz & 8.34~\phzzz & 2.5\phz\\
10000    & 10000    & 0.208\phzzz & 0.083\phzzz & 2.5\phz\\
\lastline
\end{tabular}
}
\label{tab:ode}
\end{table}

In order to check behaviour of time error estimator for variable time step (see Table \ref{tab:ode2}) we take the previous example with time step $\forall n:~0\leq n \leq N$
\begin{equation}\label{tau10}
\tau_n=\begin{cases}
0.1\tau_{\ast} ,&mod(n,2)=0\\
\tau_{\ast} ,&mod(n,2)=1
\end{cases}
\end{equation}
where $\tau_{\ast}$ is a given fixed value.  As in the case of constant time step we have the equivalence between the true error and the estimated error. We have plotted on Fig. \ref{fig:toyIndicators} evolution in time of the value $\sum_{k=0}^{n-1}{\eta}_{T}(t_k)$ compared to $e$. 

The same conclusions hold when using even more non-uniform time step $\forall n:~0\leq n \leq N$
\begin{equation}\label{tau100}
\tau_n=\begin{cases}
0.01\tau_{\ast} ,&mod(n,2)=0\\
\tau_{\ast} ,&mod(n,2)=1
\end{cases}
\end{equation}
on otherwise the same test case (see Table {\ref{tab:ode3}}).

Our conclusion is thus that for toy model classic and alternative a posteriori error estimators are sharp on both constant and variable time grids. 
\begin{figure}[h!]
    \centering
    \includegraphics[width=.85\textwidth]{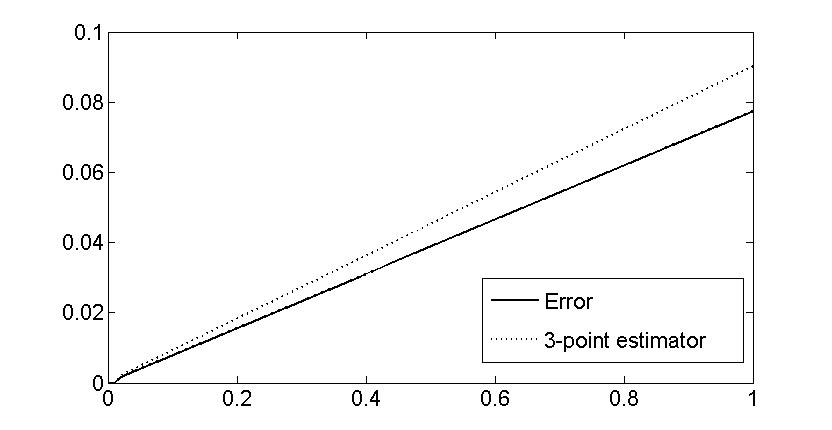}
    \caption{Evolution in time of true error and 3-point error estimate for variable time step (\ref{tau10}), $A=100$, $N=180$, $T=1$}
    \label{fig:toyIndicators}
\end{figure} 

\begin{table}[t!]
\tblcaption{Effective indices for variable time step {\rm(\ref{tau10})} and $f=0$.}
{
\begin{tabular}{@{}ccccc@{}}
\tblhead{$A$ & $N$ & $\eta_{T}$ & $e$ & $ei_{T}$ }
\noalign{\vskip 2mm} 
100\phzz & \phzz180 & 0.09~\phzzz & 0.077\phzzz & 1.17 \\
100\phzz & \phz1816 & 8.85e-04    & 7.59e-04    & 1.17 \\
100\phzz & 18180    & 8.83e-06    & 7.6e-06\phz & 1.16 \\
\noalign{\vskip 2mm} 
1000\phz & \phzz180 & 8.91~\phzzz & 7.6~~\phzzz & 1.17 \\
1000\phz & \phz1816 & 0.089\phzzz & 0.076\phzz  & 1.17 \\
1000\phz & 18180    & 8.84e-04    & 7.59e-04    & 1.16 \\
\noalign{\vskip 2mm} 
10000    & \phzz180 & 802.84\phzz & 200~~\phzzz & 4.01 \\
10000    & \phz1816 & 8.84~\phzzz & 7.58~\phzzz & 1.17 \\
10000    & 18180    & 0.088\phzzz & 0.076\phzzz & 1.16 \\
\lastline
\end{tabular}
}
\label{tab:ode2}
\end{table}

\begin{table}[t!]
\tblcaption{Effective indices for variable time step {\rm(\ref{tau100})} and $f=0$.}
{
\begin{tabular}{@{}ccccc@{}}
\tblhead{$A$ & $N$ & $\eta_{T}$ & $e$ & $ei_{T}$ }
\noalign{\vskip 2mm} 
100\phzz & \phzz196 & 0.086\phzzz & 0.084\phzzz & 1.02 \\
100\phzz & \phz1978 & 8.39e-04    & 8.26e-04    & 1.02 \\
100\phzz & 19800    & 8.38e-06    & 8.1e-06\phz & 1.03 \\
\noalign{\vskip 2mm} 
1000\phz & \phzz196 & 8.47~\phzzz & 8.26~\phzzz & 1.02 \\
1000\phz & \phz1978 & 0.083\phzzz & 0.0827\phzz & 1.02 \\
1000\phz & 19800    & 8.37e-04    & 8.26e-04    & 1.01 \\
\noalign{\vskip 2mm} 
10000    & \phzz196 & 764.2\phzzz & 200~~\phzzz & 3.82 \\
10000    & \phz1978 & 8.39~\phzzz & 8.25~\phzzz & 1.02 \\
10000    & 19800    & 0.084\phzzz & 0.083\phzzz & 1.01 \\
\lastline
\end{tabular}
}
\label{tab:ode3}
\end{table}

\subsection{The error estimator for the wave equation on structured mesh}
We now report numerical results for initial boundary-value problem for wave equation with uniform time steps when using 3-point time error estimator (\ref{in}, \ref{in0step}). We compute space estimators (\ref{space1}) and (\ref{space2}) in practice as follows:
\begin{align}
 \label{etas1} \eta_S^{(1)} (t_N) &= \max_{1 \leq n \leq N-1} \left[ \sum_{K \in \mathcal{T}_h}
   h_K^2  \left\Vert  \partial_n v_h-{f}^n_{h}
   \right\Vert_{L^2(K)}^2\right. +\left.\sum_{
E \in \mathcal{E}_h}h_{E} \|[n \cdot \nabla {u}^n_{h}]\|_{L^2(E)}^2 \right]^{1/2},
\\
\label{etas2}  \eta_S^{(2)} (t_N)  &= \sum_{n = 1}^{N-1} \tau_n \left[ \sum_{K \in \mathcal{T}_h} h_K^2
  \left\Vert  \partial_n^2 v_h - \partial_n f_h\right\Vert_{L^2(K)}^2 + \sum_{
E \in \mathcal{E}_h}h_{E} \left\Vert\left[n \cdot
  \nabla  \partial_n u_h\right]\right\Vert_{L^2(E)}^2
  \right]^{1/2}.
\end{align}
The quality of our error estimators in space and time is determined by following effectivity index:
\begin{equation*}
ei=\frac{\eta_T+\eta_S}{e}.
\end{equation*}
The true error is
\begin{equation*}
e=\max_{0 \leqslant n \leqslant N}\left(\left\Vert v^{n}_h-\cfrac{\partial u}{\partial t}(t_{n})\right\Vert_{L^2(\Omega)} ^{2}+\left\vert u^{n}_h-u(t_{n})\right\vert^{2}_{H^1(\Omega)}\right) ^{{1}/{2}}.
\end{equation*}
Consider the problem (\ref{wave}) with $\Omega=(0,1)\times(0,1),~ T=1$ and the exact solution $u$ given by
\begin{align*}
\text{case (a)}~~~ & u(x,y,t)=\cos(\pi t)\sin(\pi x)\sin(\pi y),\\
\text{case (b)}~~~ & u(x,y,t)=\cos(0.5\pi t)\sin(10\pi x)\sin(10\pi y),\\
\text{case (c)}~~~ & u(x,y,t)=\cos(15\pi t)\sin(\pi x)\sin(\pi y)
\end{align*}
We interpolate the initial conditions and the right-hand side with nodal interpolation. Structured  meshes in space (see Fig. \ref{figura1}) are used in all the experiments of this section. Numerical results are reported in Tables \ref{tab:wave1}--\ref{tab:wave3}. Note that these  cases and the meshes in space in time in the following numerical experiments are chosen so that the error in case (a) should be due to both time and space discretization, that in case (b) comes mainly from the space discretization, and that in case (c) mainly from the time discretization.

\begin{table}[t!]
\tblcaption{Results for  case (a). The quantity $N_0$ is defined in {\rm(\ref{N0})} and provided here for future reference.}
{
\begin{tabular}{@{}ccccccccc@{}}
\tblhead{$h$ & $\tau$ & $ei$ & $\eta_T$ & $\eta_S$ &$\eta^{(1)}_S$ &$\eta^{(2)}_S$ & $N_0$ & e }
\noalign{\vskip 2mm} 
\nicefrac{1}{160} & $\sqrt{h}$ & 13.74 & 0.114\phzz & 0.37\phz & 0.12\phz & 0.24\phz & 97.79 & 0.035\phz \\
\nicefrac{1}{320}  & $\sqrt{h}$ & 13.58 & 0.054\phzz  & 0.18\phz & 0.061 & 0.12\phz & 97.59 & 0.017\phz \\
\nicefrac{1}{640} & $\sqrt{h}$ & 13.42 & 0.026\phzz & 0.092 & 0.031 & 0.062 & 97.5\phz & 0.0088\\
\noalign{\vskip 2mm} 
\nicefrac{1}{160} & ${h}$ & 16.98 & 0.00062 & 0.37\phz & 0.12\phz & 0.24\phz & 97.79 & 0.021\phz \\
\nicefrac{1}{320} & ${h}$ & 16.97 & 0.00015 & 0.18\phz & 0.062 & 0.12\phz & 97.59 & 0.011\phz \\
\nicefrac{1}{640} & ${h}$ & 16.97 & 3.82e-05 & 0.092 & 0.031 & 0.062 & 97.5\phz & 0.005\phz\\
\lastline
\end{tabular}
}
\label{tab:wave1}
\end{table} 

\begin{table}[t!]
\tblcaption{Results for  case (b).}
{
\begin{tabular}{@{}cccccccc@{}}
\tblhead{$h$ & $\tau$ & $ei$ & $\eta_T$ & $\eta_S$ &$\eta^{(1)}_S$ &$\eta^{(2)}_S$ & e }
\noalign{\vskip 2mm} 
\nicefrac{1}{320} & \nicefrac{1}{20} & 13.05 & 2.03\phz & 12.15 & 6.13 & 6.02 & 1.09\\
\nicefrac{1}{320} & \nicefrac{1}{40} & 12.11 & 0.92\phz & 12.27 & 6.15 & 6.11 & 1.09\\
\nicefrac{1}{320} & \nicefrac{1}{80} & 11.62 & 0.37\phz & 12.29 & 6.16 & 6.13 & 1.09\\
\noalign{\vskip 2mm} 
\nicefrac{1}{640} & \nicefrac{1}{20} & 12.14 & 0.51\phz  & 6.09 & 3.07 & 3.02 & 0.54\\
\nicefrac{1}{640} & \nicefrac{1}{40} & 11.68 & 0.23\phz  & 6.13 & 3.08 & 3.05 & 0.54\\
\nicefrac{1}{640} & \nicefrac{1}{80} & 11.64 & 0.096 & 6.15 & 3.08 & 3.07 & 0.54\\
\lastline
\end{tabular}
}
\label{tab:wave2}
\end{table}

\begin{table}[t!]
\tblcaption{Results for case (c).}
{
\begin{tabular}{@{}cccccccc@{}}
\tblhead{$h$ & $\tau$ & $ei$ & $\eta_T$ & $\eta_S$ &$\eta^{(1)}_S$ &$\eta^{(2)}_S$ & e }
\noalign{\vskip 2mm}
\nicefrac{1}{160} & \nicefrac{1}{80} & 73.98 & 55.92 & 4.17 & 0.75\phz & 3.41 & 0.81 \\
\nicefrac{1}{320} & \nicefrac{1}{80} & 71.42 & 55.92 & 2.08 & 0.38\phz & 1.71 & 0.81\\
\nicefrac{1}{640} & \nicefrac{1}{80} & 70.13 & 55.93 & 1.04 & 0.19\phz & 0.85 & 0.81\\
\noalign{\vskip 2mm}
\nicefrac{1}{160} & \nicefrac{1}{160} & 87.44 & 14.15 & 3.78 & 0.15\phz & 3.63 & 0.21\\
\nicefrac{1}{320} & \nicefrac{1}{160} & 78.22 & 14.15 & 1.89 & 0.076 & 1.82 & 0.21\\
\nicefrac{1}{640} & \nicefrac{1}{160} & 73.61 & 14.15 & 0.95 & 0.038 & 0.91 & 0.21\\
\lastline
\end{tabular}
}
\label{tab:wave3}
\end{table}

Referring to Table \ref{tab:wave1}, we observe from first three rows that setting $h = \tau^2$ the error is divided by 2 each time $h$ is divided by 2, consistent with $e\sim O(\tau^2+h)$. The space error estimator and the time error estimator behave similarly and thus provide a good representation of the true error. The effectivity index tends to a constant value. In rows 4-6, we choose $h = \tau$ in order to insure that the discretization in time gives an error of higher order than that in space, i.e. $O(h^2)$ vs. $O(h)$, respectively. Our estimators capture well this behaviour of the two parts of the error.

In Table \ref{tab:wave2}, in order to illustrate the sharpness of the space estimator,  we take case (b) where the error is mainly due to the space discretization. We can see from this table that the space error estimator $\eta_{S}$ behaves as the true error. Indeed, for a given space step, $\eta_{S}$ does not depend on the time step $\tau$, and for constant $\tau$, $\eta_{S}$ is divided by two when the space step $h$ is divided by two.

Finally, we consider case (c), Table \ref{tab:wave3}. We observe that the time error estimator $\eta_{T}$ behaves as the true error, when the error is mainly due to the time discretization.

We therefore conclude that our time and space error estimators are sharp in the regime of constant time steps and structured space meshes. They separate well the two sources of the error and can be thus used for the mesh adaptation in space and time.
\begin{remark}
As said already, the space estimator $\eta_S$ behaves as $O(h)$ in  the numerical experiments reported in Tables \ref{tab:wave1}-\ref{tab:wave2}. The situation is slightly different in Table \ref{tab:wave3}. Indeed, the first part of space error estimator $\eta^{(1)}_S$ behaves here as $O(\tau^2h)$. This can be explained by the fact that, as seen from the definitions (\ref{etas1})--(\ref{etas2}), both $\eta_S^{(1)}$ and $\eta_S^{(2)}$ are also influenced by discretization in time. In general, in the leading order in $h$ and $\tau$, one can conjecture $\eta_S^{(1,2)}=Ah+Bh\tau^2$ with case dependent $A$ and $B$. The second term $Bh\tau^2$ is asymptotically negligible but it can become visible in some situations where the solution is highly oscillating in time and the mesh in time is not sufficiently refined, as indeed observed with $\eta^{(1)}_S$ in Table \ref{tab:wave3}. Fortunately, its value is small compared to the time error estimator and thus we can hope that this effect is not essential for mesh refinement.
\end{remark}

\subsection{The error estimator for the wave equation on unstructured mesh}\label{unstr}

\begin{figure}[h]
\noindent\centering{
$\begin{matrix}
~~\includegraphics[width=50mm]{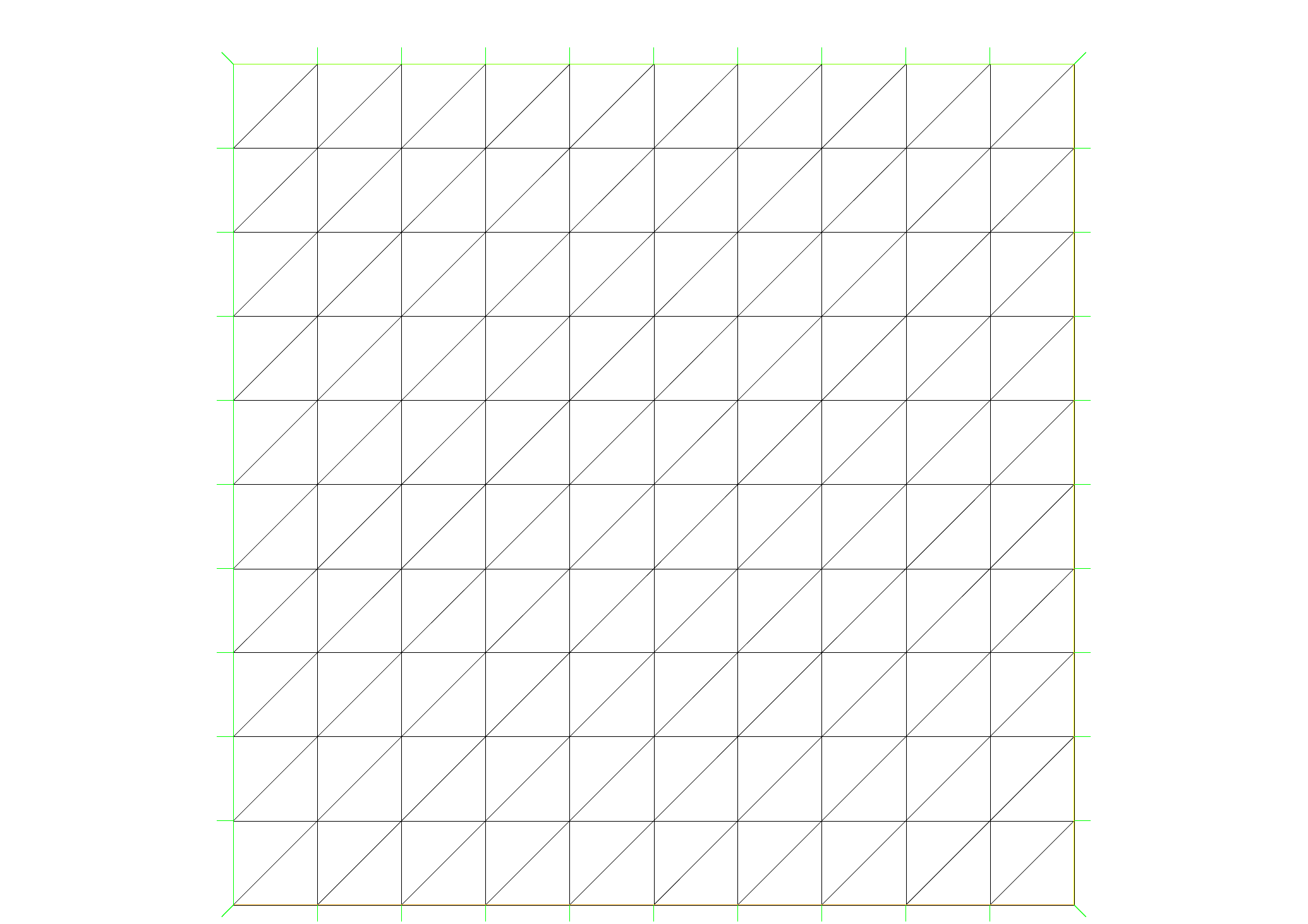}~~&~~\includegraphics[width=50mm]{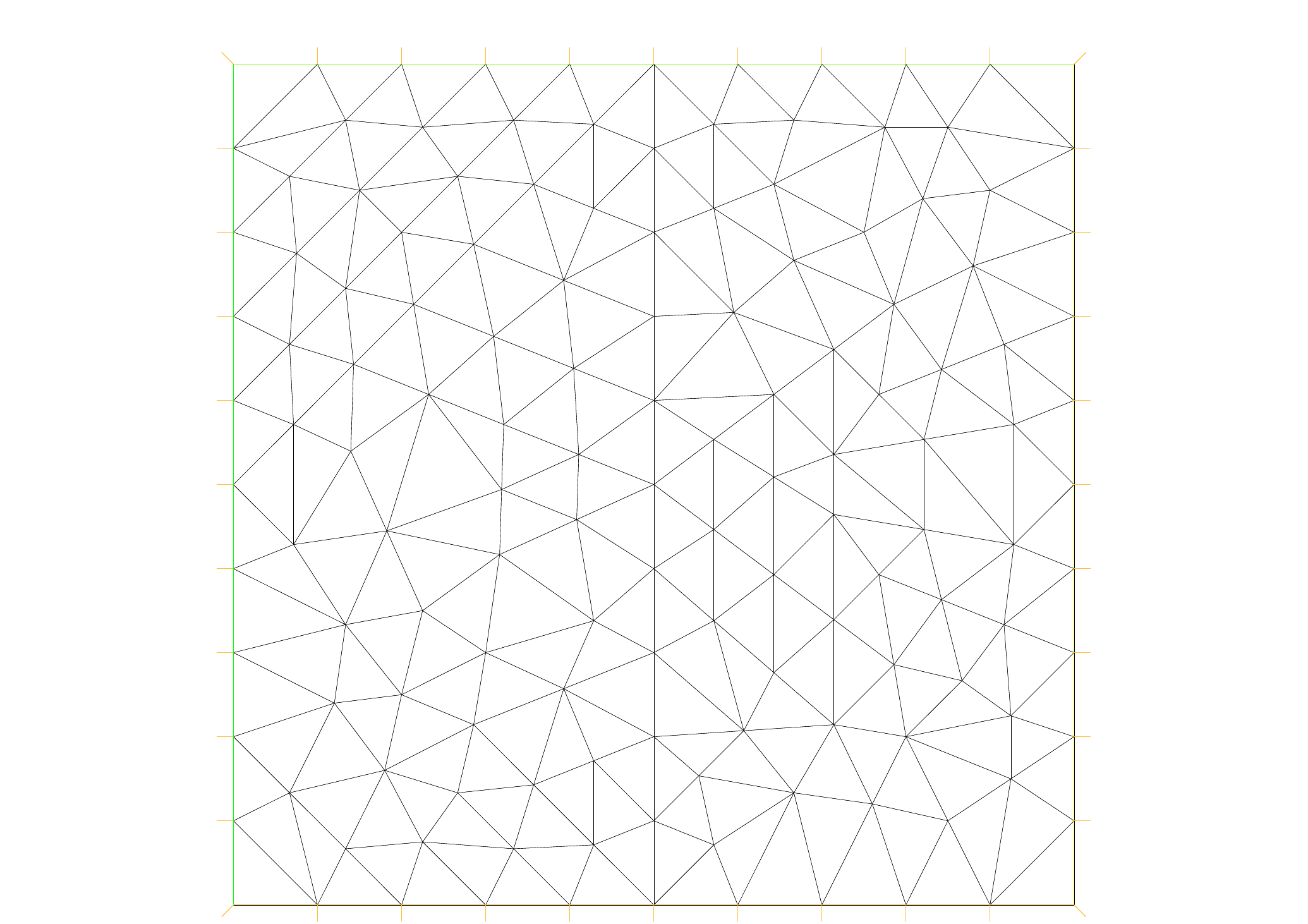}
\end{matrix}$
}
\caption{Structured (on the left) and unstructured (on the right) a $10\times10$ meshes of the unit square.}
\label{figura1}
\end{figure}

We turn now to the numerical experiments on unstructured Delaunay meshes, cf.  Fig. \ref{figura1} (right). These experiments will reveal the dependence of the error estimators on approximation of initial conditions and of the right-hand side $f$. Indeed, as noted in Subsection \ref{optimality}, these approximations should be chosen  carefully  to ensure the optimality of our error estimators.

\begin{figure}[h]
\noindent\centering{
$\begin{matrix}
~~\includegraphics[width=52mm]{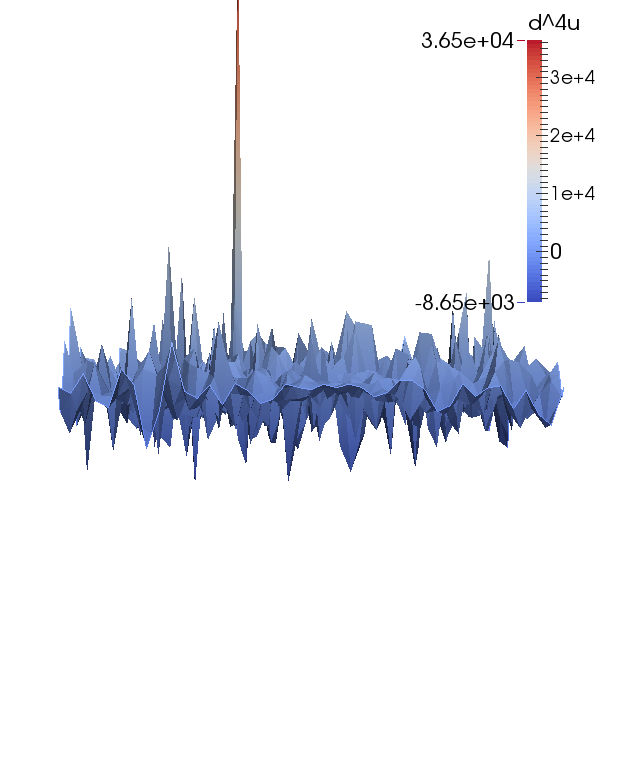}&\includegraphics[width=50mm]{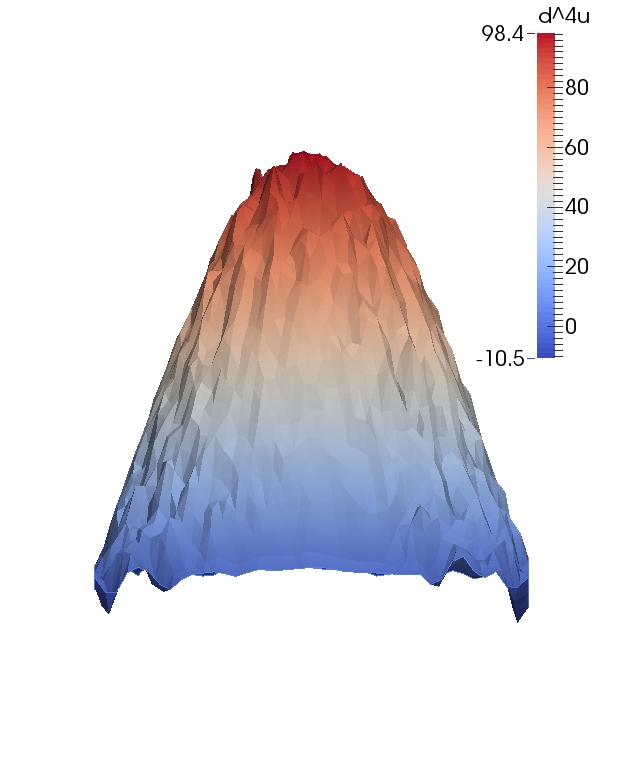}
\end{matrix}$
}
\caption{${\partial}_4^4 u_h$ for different discretization of the initial conditions: on the left (see Table \ref{tab:wave5}) we take $u^0_h$  as the nodal interpolation of $u_0$ while on the right (see Table \ref{tab:wave6}) $u^0_h=\Pi_h u_0$ , $h=0.125$, $\tau= 0.025$}
\label{figura2}
\end{figure}

\begin{table}[t!]
\tblcaption{Results for  case (a), constant time steps, unstructured Delaunay meshes, nodal interpolation of the initial conditions and $f$ as in {\rm(\ref{NodalInit})}.}
{
\begin{tabular}{@{}ccccccccc@{}}
\tblhead{$h$ & $\tau$ & $ei$ & $\eta_T$ & $\eta_S$ &$\eta^{(1)}_S$ &$\eta^{(2)}_S$& $N_0$ & e}
\noalign{\vskip 2mm} 
\nicefrac{1}{160} & $\sqrt{h}$ & 75\phzzz  & 2.1\phz & 0.33 & 0.094 & 0.23\phz & 934718\phz & 0.033\phz \\
\nicefrac{1}{320}  & $\sqrt{h}$ & 120.74 & 1.76 & 0.17 & 0.047 & 0.13\phz & 3.31e+06 & 0.016\phz \\
\nicefrac{1}{640} & $\sqrt{h}$ & 244.56 & 1.89 & 0.11 & 0.023 & 0.082 & 1.44e+07 & 0.0082\\
\noalign{\vskip 2mm} 
\nicefrac{1}{160} & ${h}$ & 196.92 & 1.61 & 1.73 & 0.096 & 1.63\phz & 934718\phz & 0.017\phz \\
\nicefrac{1}{320} & ${h}$ & 353.63 & 1.43 & 1.49 & 0.047 & 1.45\phz & 3.31e+06& 0.088\phz \\
\nicefrac{1}{640} & ${h}$ & 751.43 & 1.54 & 1.59 & 0.023 & 1.56\phz & 1.44e+07 & 0.0042\\
\lastline
\end{tabular}
}
\label{tab:wave5}
\end{table}

\begin{table}
\tblcaption{Results for  case (a), constant time steps, unstructured Delaunay meshes, orthogonal projection of the initial conditions and $f$ as in {\rm(\ref{Projinit})}.}
{
\begin{tabular}{@{}ccccccccc@{}}
\tblhead{$h$ & $\tau$ & $ei$ & $\eta_T$ & $\eta_S$ &$\eta^{(1)}_S$ &$\eta^{(2)}_S$& $N_0$ & e}
\noalign{\vskip 2mm}
\nicefrac{1}{160} & $\sqrt{h}$ & 12.29  & 0.115\phzz & 0.28\phz & 0.094 & 0.19\phz & 98.48 & 0.032\phz \\
\nicefrac{1}{320}  & $\sqrt{h}$ & 12.13 & 0.054\phzz & 0.14\phz & 0.047 & 0.094 & 98.18 & 0.016\phz \\
\nicefrac{1}{640} & $\sqrt{h}$ & 12.\phzz & 0.027\phzz & 0.071 & 0.024 & 0.047 & 98.27 & 0.0081\\
\noalign{\vskip 2mm}
\nicefrac{1}{160} & ${h}$ & 17.4\phz & 0.00062 & 0.29\phz & 0.095 & 0.19\phz & 98.48 & 0.017\phz \\
\nicefrac{1}{320} & ${h}$ & 17.25 & 0.00015 & 0.14\phz & 0.047 & 0.094 & 98.18 & 0.082\phz \\
\nicefrac{1}{640} & ${h}$ & 17.28 & 3.83e-05 & 0.071 & 0.023 & 0.047 & 98.27& 0.0041\\
\hline
\hline
\end{tabular}
}
\label{tab:wave6}
\end{table}

We consider the test case from the previous subsection with the exact solution $u$ given by case (a).  We test two different ways to approximate the initial conditions and the right-hand side:  nodal interpolation 
\begin{equation}\label{NodalInit}
u^0_h = I_h u^0,~v^0_h = I_h v^0,~f^n_h= I_h f^n,~0\leq n \leq N
\end{equation}
and orthogonal projections as in Lemma \ref{bound_d4u}
\begin{equation}\label{Projinit}
u^0_h = \Pi_h u^0,~v^0_h = \Pi_h v^0,~f^n_h= P_h f^n,~0\leq n \leq N.
\end{equation}
The results are reported in Tables \ref{tab:wave5} and \ref{tab:wave6}. 
The meshes, the time steps and other details of the numerical algorithm,  are exactly the same in these two tables. We observe that the errors are very similar as well and conclude therefore that the accuracy of the method does not depend on the manner in which the initial conditions and $f$ are approximated, either (\ref{NodalInit}) or (\ref{Projinit}). 

On the contrary, the behaviour of error estimators is quite different in the two cases. From Table \ref{tab:wave5} (nodal interpolation), we see that the time error estimator $\eta_T$ blows up with mesh refinement, while the second part of the space estimator $\eta^{(2)}_{S}$ behaves (non optimally) like $O(\tau+h)$.  Only the first part of the space estimator $\eta^{(1)}_{S}$ behaves as the true error. Such a strange behaviour of our estimators indicates the unboundedness of higher order discrete derivatives in time.  Indeed, the estimators $\eta_T$ and $\eta^{(2)}_{S}$ contain high order discrete derivatives ${\partial}^2_n f_h-A_h{\partial}^2_n u_h$ and ${\partial}^2_n v_h$ respectively. These error estimators can be of the optimal order only if all these derivatives are uniformly bounded. We recall that this property was examined in Lemma \ref{bound_d4u} and its proof hinges on the boundedness of 
\begin{equation}\label{N0}
N_0=\left\Vert A^2_hu^0_h-A_hf^0_h\right\Vert_{L^2(\Omega)}.
\end{equation}
However, as reported in Table \ref{tab:wave5}, $N_0$ also blows up under the nodal interpolation of initial conditions and of the right-hand side. This is not surprising given that the boundedness of $N_0$ in Lemma  \ref{bound_d4u} is a consequence of Lemma \ref{bound_AhPhu} and thus it is not guaranteed if one replaces projections (\ref{Projinit}) by nodal interpolation (\ref{NodalInit}). On the other hand, the results in Table \ref{tab:wave6} corresponding to interpolation by projection (\ref{Projinit}) confirm the order $O(\tau^2+h)$ for our error estimators, consistently with the theory developed in Lemmas  \ref{bound_d4u} and \ref{bound_AhPhu}.

\begin{table}[t!]
\tblcaption{Results for  case (a), constant time step, unstructured Delaunay mesh, orthogonal projection of the initial conditions and $f$ as in {\rm(\ref{Projinit})}, $M_1=\Vert A_hP_hu^0\Vert_{L^2(\Omega)}$, $M_2=\Vert P_hu^0\Vert_{H^1(\Omega)}$.}
{
\begin{tabular}{@{}ccccccc@{}}
\tblhead{Mesh & $\tau$ & $ei$ & $M_1$ & $M_2$  & $N_0$ & e}
\noalign{\vskip 2mm}
case (1) & \nicefrac{1}{10} & 17.15 & 10.39& 2.25 & 102.59 & 0.37\phz \\
case (2) & \nicefrac{1}{20} & 17.15 & 9.99 & 2.22 & 98.62\phz & 0.099 \\
case (3)  & \nicefrac{1}{40} & 17.15 & 9.97 & 2.22 & 98.45\phz & 0.025\\
\lastline
\end{tabular}
}
\label{tab:wave7} 
\end{table}

\begin{figure}[h]
\noindent\centering{
$\begin{matrix}
~~\includegraphics[width=40mm]{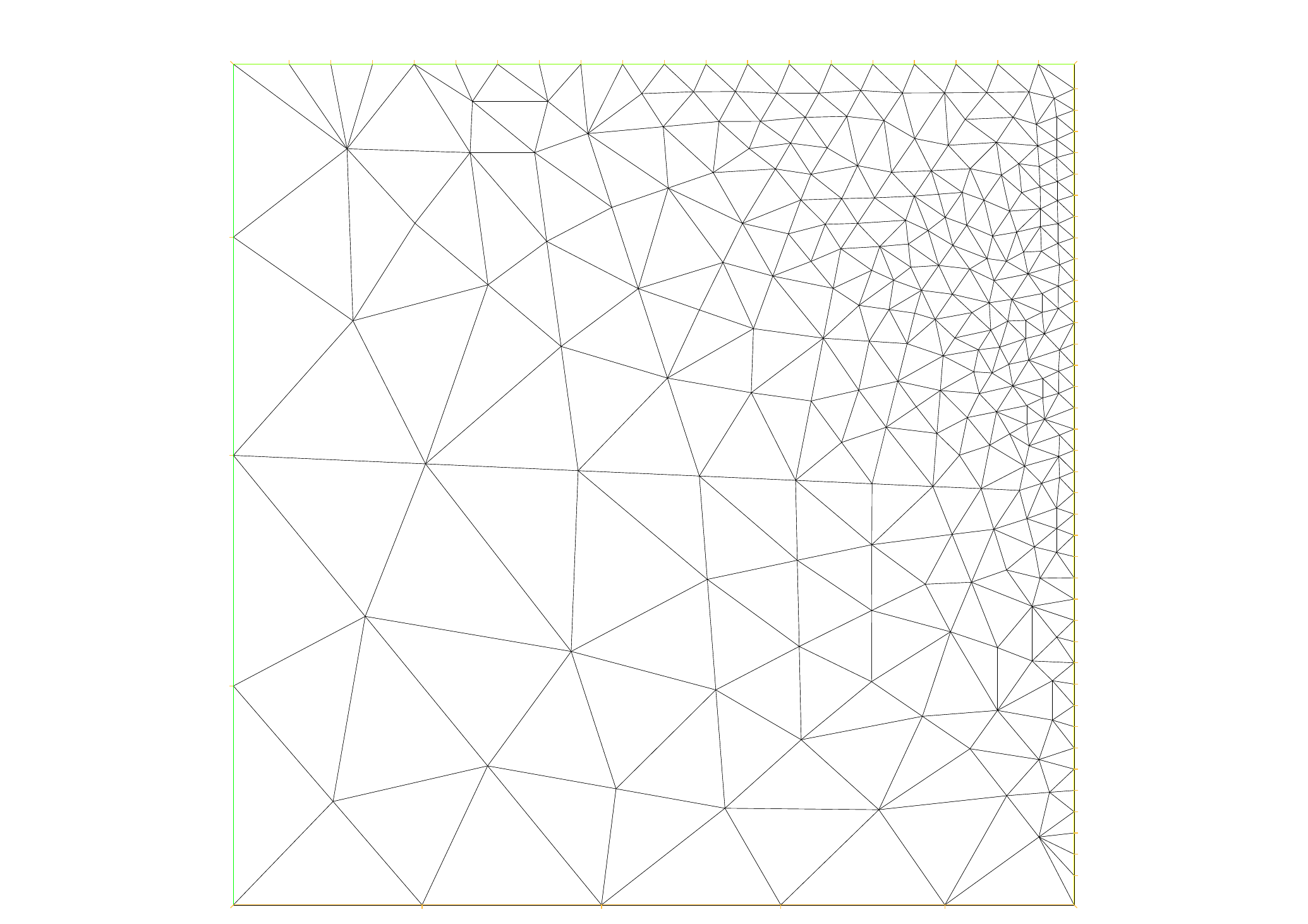}&\includegraphics[width=40mm]{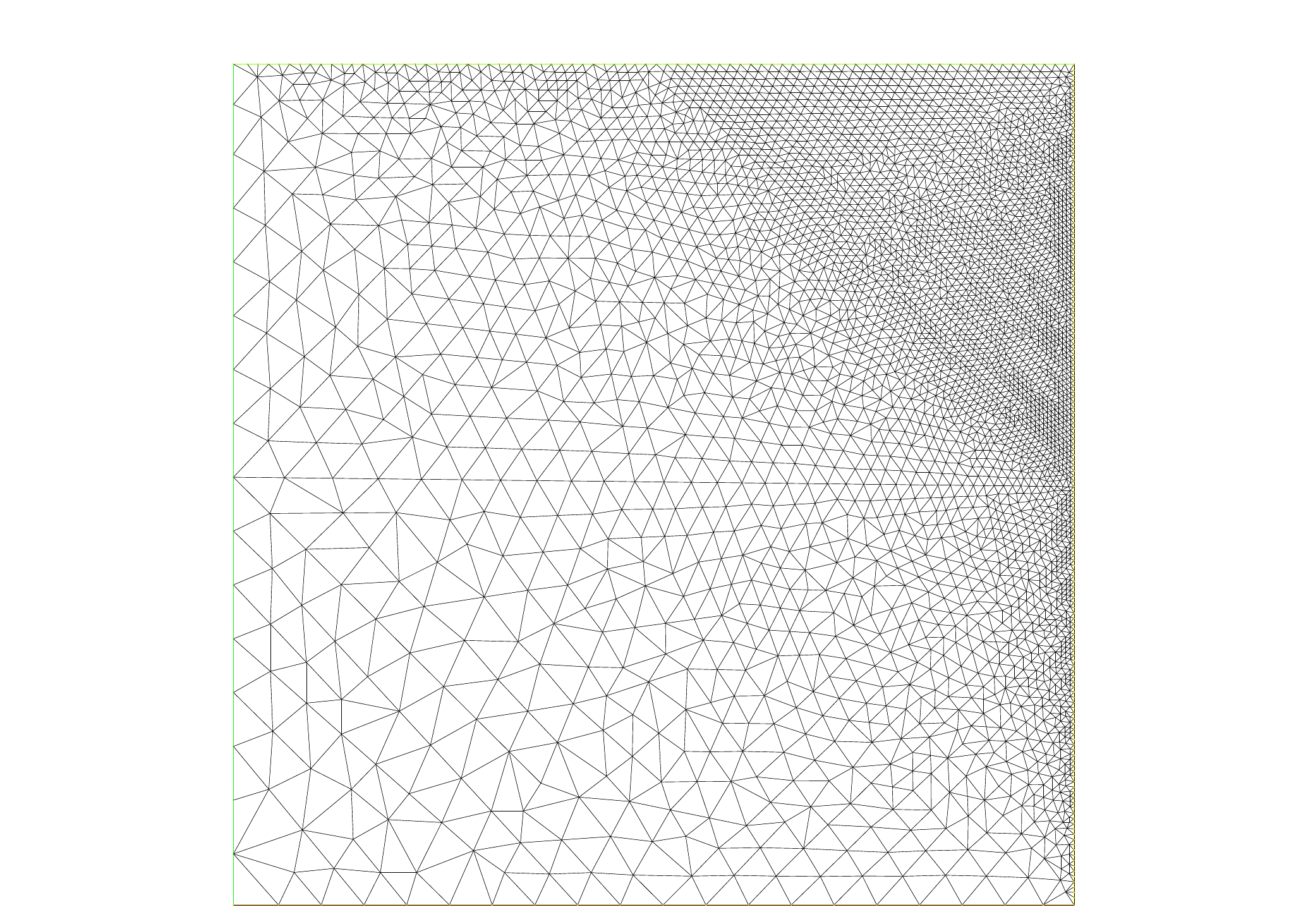}&\includegraphics[width=40mm]{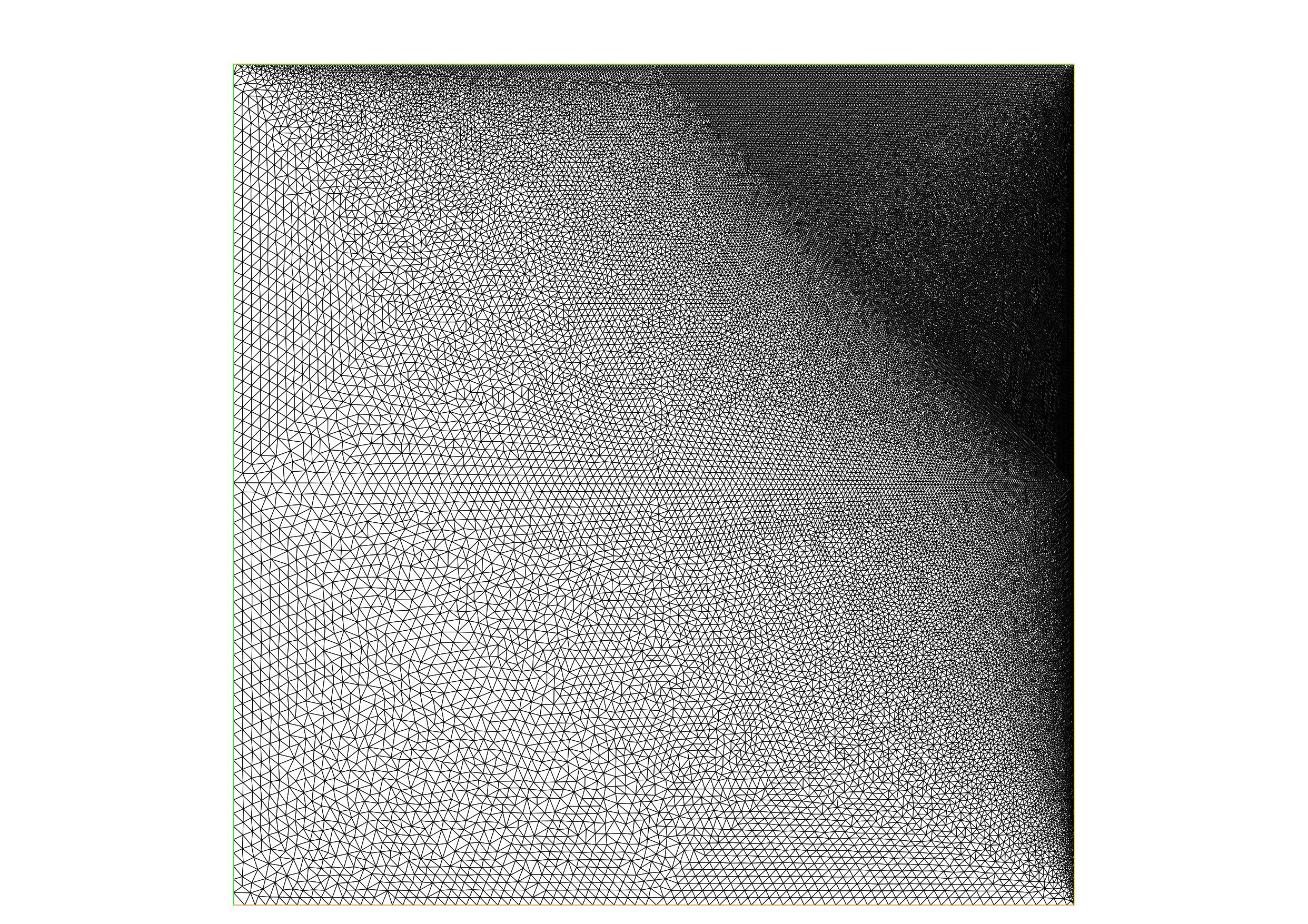}\\
\text{case (1), 486 triangles}&\text{case (2), 7535 triangles}&\text{case (3), 121299 triangles}
\end{matrix}$
}
\caption{non quasi-uniform meshes (see Table \ref{tab:wave7})}
\label{figura3}
\end{figure}

  The huge difference between the two data approximations can be also seen by looking directly at ${\partial}^4_4 u_h$. We report this quantity in  Fig. \ref{figura2}  for the case (a) on a mesh with $h=0.0125$ and time step $\tau=0.025$ at $t=t_4=0.1$. On the left picture (nodal interpolation) we see that ${\partial}^4_4 u_h$ contains a lot of severe spurious oscillations, while the right picture (projection of initial conditions) contains a reasonable and quite smooth approximation of $\cfrac{\partial^4 u}{\partial t^4}$. This is another manifestation of the critical importance of the choice of an approximation of initial conditions and of the right-hand side for our error estimators. We note that such a phenomenon was not observed for the heat equation in \cite{LPP}. We also recall from Table \ref{tab:wave1} that space and time error estimators provide a good representation of the true error on a structured mesh even under the nodal interpolation. Note that the quantity defined by (\ref{N0}) remains also bounded on the structured mesh.

We recall that the theory of Subsection \ref{optimality}, in particular Lemma \ref{bound_AhPhu}, are established under the quasi-uniform meshe assumption. We conclude this article by a numerical test on non quasi-uniform meshes in order to asses the stability of operators $A_h$ and $P_h$. We apply our numerical method to (\ref{wave}) with the exact solution $u$ from case (a) on meshes from  Fig. \ref{figura3}. The results are given in Table \ref{tab:wave7}.  We see that space and time error estimators provide a good representation of the true error, like in examples from Tables \ref{tab:wave1} and \ref{tab:wave6} with quasi-uniform meshes. Moreover, we observe stability for terms $\Vert A_hP_hu^0\Vert_{L^2(\Omega)}$, $\Vert P_hu^0\Vert_{H^1(\Omega)}$, and consequently $N_0$. This indicates that our error indicators may be useful for time and space adaptivity on rather general meshes.

\section{Conclusions}
An a posteriori error estimate in the $L^{\infty}$-in-time/energy-in-space norm is proposed for the wave equation discretized by the Newmark scheme in time and the finite element method in space.  Its reliability is proven theoretically in Theorem \ref{lemest3}. Moreover, numerical experiments show its effectivity. Our estimators are designed  to separate the error coming from discretization in space and that in time and should be therefore useful for time and space adaptivity. We have demonstrated, both theoretically and experimentally, the critical importance of the manner in which the initial conditions and the right-hand side are approximated. Indeed, under nodal interpolation the scheme in itself produces optimal results, bur certain quantities in a posteriori error estimates can blow up with mesh refinement. The remedy for this problem consists in using   orthogonal pro1jections for initial conditions and the right-hand side, cf. Lemma \ref{bound_d4u}. 

\pagebreak

\bibliographystyle{IMANUM-BIB}
\bibliography{apost}

\end{document}